\documentclass[english,10pt]{amsart}

\usepackage[english]{babel}

\usepackage{amscd,amssymb,amsfonts,amsmath}
\usepackage{tikz-cd}

\usepackage{bm}
\usepackage{graphics}
\usepackage{epsfig}
\usepackage{mathrsfs}
\usepackage{xcolor}

\DeclareMathAlphabet{\mathscrbf}{OMS}{mdugm}{b}{n}

\definecolor{violet}{rgb}{0.0,0.2,0.7}
\definecolor{rouge2}{rgb}{0.8,0.0,0.2}
\usepackage{hyperref}
\hypersetup{
    unicode=false,          
    pdftoolbar=true,        
    pdfmenubar=true,        
    pdffitwindow=false,     
    pdfstartview={FitH},    
    pdftitle={},    
    pdfauthor={},     
    colorlinks=true,       
   linkcolor=violet,          
    citecolor=rouge2,        
    filecolor=black,      
    urlcolor=cyan}           

\usepackage[text={6.7in,9.2in},centering]{geometry}

\makeatletter
\renewcommand\subsection{\@startsection{subsection}{2}%
  \z@{.5\linespacing\@plus.7\linespacing}{-.5em}%
  {\normalfont\sffamily}}  
\makeatother

\newcommand{\nablaBott}{\nabla_{\! \textup{B}}}

\renewcommand{\phi}{\varphi}

\newcommand{\wt}{\widetilde}

\renewcommand{\le}{\leqslant}
\renewcommand{\ge}{\geqslant}

\newcommand{\cF}{\mathcal{F}}
\newcommand{\cG}{\mathcal{G}}
\newcommand{\cH}{\mathcal{H}}

\newcommand{\cL}{\mathcal{L}}

\newcommand{\sA}{\mathscr{A}}

\newcommand{\sD}{\mathscr{D}}
\newcommand{\sE}{\mathscr{E}}
\newcommand{\sF}{\mathscr{F}}
\newcommand{\sG}{\mathscr{G}}

\newcommand{\sL}{\mathscr{L}}
\newcommand{\sM}{\mathscr{M}}
\newcommand{\sN}{\mathscr{N}}
\newcommand{\sO}{\mathscr{O}}

\newcommand{\dd}{\textup{d}}

\newtheorem{thm}{Theorem}[section]

\newtheorem{question}[thm]{Question}
\newtheorem{lemma}[thm]{Lemma}
\newtheorem{cor}[thm]{Corollary}

\newtheorem{prop}[thm]{Proposition}

\newtheorem*{thm*}{Theorem}
\theoremstyle{definition}
\newtheorem{defn}[thm]{Definition}

\newtheorem{defn-thm}[thm]{Definition-Theorem} 
\newtheorem{defn-lemma}[thm]{Definition-Lemma}

\theoremstyle{remark}

\newtheorem{fact}[thm]{Fact}

\newtheorem*{not-and-def}{Notation and definitions}
 
\newtheorem{rem}[thm]{Remark}

\newtheorem{exmp}[thm]{Example}

\numberwithin{equation}{section}

\def\factor#1.#2.{\left. \raise 2pt\hbox{$#1$} \right/\hskip -2pt\raise -2pt\hbox{$#2$}}

\begin{document} 

\title[Transversely projectively flat foliations]{Transversely projectively flat foliations}

\author{St\'ephane \textsc{Druel}}

\address{St\'ephane Druel: CNRS, Université Claude Bernard Lyon 1, UMR 5208, Institut Camille Jordan, F-69622 Villeurbanne, France} 

\email{stephane.druel@math.cnrs.fr}


\begin{abstract}
We study transversely projectively flat regular foliations on compact Kähler manifolds.
\end{abstract}

\maketitle

{\small\tableofcontents}

\section{Introduction}
Let $X$ be a compact Kähler manifold with flat tangent bundle. As a classical consequence of the existence of a 
K\"ahler-Einstein metric on $X$, proved independently by Aubin (\cite{aubin_KE}) and Yau (\cite{Yau_KE}), $X$ is then a finite \'etale quotient of a complex torus (see \cite[Chapter IV Corollary 4.15]{kobayashi_diff_geom_vb}). 

\medskip

In this article, we partially extend the above result to transversely projectively flat regular foliations. We refer to Section \ref{section:definition} for this notion. In a nutshell, a regular foliation $\cF$ on $X$ is transversely flat (resp. transversely projectively flat) if the tangent bundle of the space of leaves of $\cF$ is 
flat (resp. projectively flat). The simplest examples 	are linear foliations on complex tori, and foliations
on compact Kähler manifolds defined by pointwise linearly independent holomorphic $1$-forms with values in flat line bundles. Also, the pull-back of a transversely projectively flat foliation  via a submersion is transversely projectively flat as well.

\medskip

First, we address transversely projectively flat foliations with compact leaves.

\begin{thm}\label{thm_intro:compact_leaves}
Let $X$ be a compact K\"ahler manifold, and let $\cF$ be regular foliation of codimension at least $2$ on $X$. Suppose that $\cF$ is transversely projectively flat with compact leaves. Then there exist a complex torus $T$ and a smooth morphism $f \colon Y \to T$ as well as a finite \'etale cover $\gamma\colon Y \to X$ such that $\gamma^{-1}\cF$ is induced by $f$. 
\end{thm}

The proof of Theorem \ref{thm_intro:compact_leaves} above makes use of the following characterization of torus quotients. Theorem \ref{thm_intro:projectively_flat_spaces} below extends \cite[Theorem D]{CGGN22} to projectively flat compact K\"{a}hler varieties.

\begin{thm}\label{thm_intro:projectively_flat_spaces}
Let X be a compact K\"{a}hler variety of dimension at least $2$ with klt singularities. Suppose that $T_{\textup{reg}}$ is projectively flat. Then there exists a complex torus $T$, and a quasi-\'etale cover $T \to X$.
\end{thm}

Finally, we describe the structure of transversely projectively flat foliations of dimension $1$ on complex projective manifold.

\begin{thm}\label{thm_intro:dimension_one}
Let $\cF$ be regular foliation of dimension $1$ and codimension at least $2$ on a complex projective manifold $X$. Suppose that $\cF$ is transversely projectively flat. Then one of the following holds.
\begin{enumerate}
\item The foliation $\cF$ is induced by a $\mathbb{P}^1$-bundle structure $f\colon X \to Y$ onto a finite \'etale quotient of an abelian variety.
\item There exists an abelian variety $A$ as well as a finite \'etale morphism $\gamma\colon A \to X$ such that $\gamma^{-1}\cF$ is a linear foliation.
\item There exist a smooth projective curve $C$ and an abelian variety $B$ as well as a finite \'etale mrophism $\gamma\colon C \times B \to X$ such that $\gamma^{-1}\cF$ induces a flat Ehresmann connection on the projection morphism $\textup{pr}_C \colon C \times B \to C$.
\end{enumerate}
\end{thm}

Note that flat vector bundles are projectively flat with vanishing Chern classes but the converse statement does not hold in general. As an intermediate step, we obtain the following result.

\begin{thm}\label{thm_intro:vanishing}
Let $X$ be a complex manifold, and let $\cF$ be regular foliation of codimension $q \ge 2$ with normal bundle $\sN$. If $\cF$ is transversely projectively flat, then $\sN$ admits a connection. In particular, we have $c_i(\sN)=0 \in H^i(X,\Omega_X^i)$ for any index $i \ge 1$.
\end{thm}

\begin{rem}
In \cite{jahnke_radloff_13}, Jahnke and Radloff proved that finite étale quotients of abelian varieties are the only complex projective manifolds with numerically projectively flat tangent bundle. Greb, Kebekus, and Peternell then generalised the theorem of Jahnke and Radloff to projective varieties with klt singularities in \cite{GKP_proj_flat_JEP}. Theorem \ref{thm_intro:vanishing} makes it possible to considerably simplify the proof of the results mentioned above. 
\end{rem}

\subsection*{Acknowledgements} The author would like to thank Jorge V. Pereira for very helpful discussions.

\section{Notation, conventions and standard facts}\label{section:notation}

\subsection{Global conventions} Throughout the paper, all varieties are assumed to be defined over the field of complex numbers. We will freely switch between the algebraic and analytic context. 

\subsection{Foliations} 
Let $X$ be a complex manifold. By a \emph{regular foliation} $\cF$ of \emph{dimension} $p$ on $X$, we mean an involutive subbundle $T_\cF \subseteq T_X$ of rank $p$. Its \emph{codimension} is $q:=\dim X - p$. We will refer to $T_\cF$ as the \emph{tangent bundle} of $\cF$. The \emph{normal bundle} of $\cF$ is $\sN_\cF:=T_X / T_\cF$; by definition, it is a vector bundle of rank $q$. 

\medskip

A \emph{leaf} of $\cF$ is a maximal connected and immersed holomorphic submanifold $L \subseteq X$ such that
$T_L=T_\cF|_L$.

\medskip

The $q$-th wedge product of the inclusion $\sN^\vee\subseteq \Omega^1_X$ gives rise to a nowhere vanishing global section $\omega\in H^0(X,\Omega^{q}_X\otimes \det\sN)$. Moreover, $\omega$ is \emph{locally decomposable} and \emph{integrable}. To say that $\omega$ is locally decomposable means that, in a neighborhood of every point of $X$, $\omega$ decomposes as the wedge product of $q$ local $1$-forms $\omega=\omega_1\wedge\cdots\wedge\omega_q$.
To say that it is integrable means that for this local decomposition one has $\dd\omega_i\wedge \omega=0$ for every  $i\in\{1,\ldots,q\}$. The integrability condition for $\omega$ is equivalent to the condition that $T_\cF$ 
is closed under the Lie bracket. Conversely, let $\sL$ be a line bundle on $X$, and let $\omega\in H^0(X,\Omega^{q}_X\otimes \sL)$ be a nowhere vanishing global section. Suppose that $\omega$ is locally decomposable and integrable. Then the kernel of the morphism $T_X \to \Omega^{q-1}_X\otimes \sL$ given by the contraction with $\omega$ defines a regular foliation of codimension $q$ on $X$. These constructions are inverse of each other. 

\medskip

Let $f \colon X \to Y$ be a submersion between complex manifolds, and let $\cF$ be a regular foliation of codimension $q$ on $Y$. The \emph{pull-back} $f^{-1}\cF$ of $\cF$ via $f$ is the regular foliation of codimension $q$ on $X$ whose tangent bundle is $T_{f^{-1}\cF}:=\dd f^{-1}(T_\cF)$. As a consequence, we have $\sN_{f^{-1}\cF}\cong f^*\sN_\cF$.

\medskip

Let $f \colon X \to Y$ be a surjective submersion between complex manifolds, and let $\cF$ be a foliation of dimension $p$ on $X$. We say that $\cF$ is \emph{projectable} under $f$ if, for each point $y \in Y$, $\dd f_x(T_{\cF,x})$ is independent of the choice of $x \in f^{-1}(y)$ and $\dim \dd f_x (T_{\cF,x}) = p$. Then $\cF$ induces a regular foliation $\cG$ of dimension $p$ on $Y$ such that $T_\cF \cong f^*T_\cG$ (see \cite[Paragraph 2.7]{druel_bbcd2}).

\subsection{Connections}\label{paragraph:connections} 
Let $X$ be a complex manifold, and let $\sE$ be a vector bundle on $X$, \emph{i.e.} a locally a free $\sO_X$-module of finite rank. Let $r$ be the rank of $\sE$.

A \emph{connection} $\nabla$ on $\sE$ is a morphism of abelian sheaves
\[
\nabla\colon \sE \to \Omega_X^1\otimes\sE
\] 
satisfying 
\[
\nabla(fe)=f\nabla(e)+\dd f \otimes e
\] 
for any local sections $f$ of $\sO_X$ and $e$ of $\sE$ over some open subset of $X$. The \emph{curvature} of $\nabla$ is the $\sO_X$-linear map
\[
\textup{K}(\nabla):=\nabla^1\circ\nabla\colon \sE \to \Omega_X^2\otimes \sE
\]
of $\sO_X$-modules, where $\nabla^i$ denotes the extension of $\nabla$ to 
a morphism of abelian sheaves $\Omega_X^i \otimes \sE \to \Omega_X^{i+1} \otimes \sE$ given by 
\[
\nabla^i(\alpha\otimes e)= \dd \alpha \otimes e + (-1)^i \alpha \wedge \nabla(e)
\]
for any local sections $\alpha$ of $\Omega_X^i$ and $e$ of $\sE$ over some open subset of $X$. 
The connection is said to be \textit{flat} if $\textup{K}(\nabla)=0$.

\medskip

We say that $\sE$ is \emph{flat} if $\sE$ admits a flat connection.
 
\begin{fact}
The vector bundle $\sE$ is flat if and only if $\mathbb{V}(\sE^\vee):=\textup{Spec}_X\textup{S}^\bullet(\sE^\vee)$ comes from a representation $\pi_1(X) \to \textup{GL}(r,\mathbb{C})$.
\end{fact}

\begin{fact}[{\cite[Proposition II.3.1(a)]{kobayashi_diff_geom_vb}}]\label{fact:connection_vanishing_chern_classes}
If $\sE$ is flat, then $c_i(\sE)=0$ for any index $i \ge 1$. 
\end{fact}

\medskip

We shall now briefly recall some basic operations on pairs $(\sE,\nabla)$ consisting of a vector bundle $\sE$ on $X$ and a connection $\nabla$ on $\sE$.

\medskip

Given connections induce new connections on associated vector bundles. The direct sum and tensor product of any two vector bundles equipped with connections (resp. flat connections) have natural connections (resp. flat connections). Also, the dual and the exterior powers of a vector bundle with a connection are equipped with natural connections (resp. flat connections) as well. 

\medskip

Let $f \colon X \to Y$ be a holomorphic map between complex manifolds, and let $\sE$ be a vector bundle on $Y$ with a connection $\nabla$. Let $\eta\colon f^*\Omega_Y^1 \to \Omega_X^1$ be the canonical morphism. The map 
\[
f^*\nabla \colon f^*\sE = \sO_X \otimes_{f^{-1}\sO_Y} f^{-1}\sE \to \Omega_X^1 \otimes f^*\sE 
\]
defined by 
\[
(f^*\nabla) (h \otimes f^{-1}e) = h ((\eta \otimes \textup{id})(f^*(\nabla (e)))) + \dd h \otimes f^{-1}e 
\]
for any local section $e$ of $\sE$ over some open subset $U$ of $Y$ and any holomorphic function $h$ on some open subset $V$ of $X$ such that $f(V) \subseteq U$ is well-defined and induces a connection $f^*\nabla$ on $f^*\sE$. If $\nabla$ is flat, then $f^*\nabla$ is flat as well. 

\subsection{Projective connections} A central notion of this article is that of a flat projective connection in the sense of Beilinson and Kazhdan. 

Let $X$ be a complex manifold, and let $\sE$ be a vector bundle of rank $r$ on $X$. Let $\sD^1(\sE) \subset \sE\hspace{-0.07cm}\textit{nd}_{\,\mathbb{C}_X}(\sE)$ be the sheaf of differential operators $\Delta\colon \sE \to \sE$ of degree at most $1$ with scalar symbol $\sigma(\Delta)$. A connection $\nabla$ on $\sE$ can alternatively be defined as an $\sO_X$-linear map
\[
\nabla \colon T_X \longrightarrow  \sE\hspace{-0.07cm}\textit{nd}_{\,\mathbb{C}_X}(\sE)
\]
satisfying the Leibnitz rule 
\[
\nabla (v) (fe)=f\nabla (v) (e)+v(f) e
\] 
for any local sections $v$ of $T_X$, $f$ of $\sO_X$ and $e$ of $\sE$ over some open subset of $X$. In other words, a connection on $\sE$ is simply a splitting of the Atiyah exact sequence
\begin{center}
\begin{tikzcd}
0 \to \sE\hspace{-0.07cm}\textit{nd}_{\,\sO_X}(\sE) \ar[r] & \sD^1(\sE) \ar[r, "\sigma"'] & T_X \ar[r]\ar[l, bend right, "\nabla"'] & 0.
\end{tikzcd}
\end{center}
Moreover, the connection $\nabla$ is flat if and only if the map $T_X \to \sE\hspace{-0.07cm}\textit{nd}_{\,\mathbb{C}_X}(\sE)$ is a morphism of Lie algebras.

\medskip

A \textit{projective connection} $\square$ on $\sE$ is a splitting of the push-out 
\begin{center}
\begin{tikzcd}
0 \to \sE\hspace{-0.07cm}\textit{nd}_{\,\sO_X}(\sE)/\sO_X \textup{id}\ar[r] & \sD^1(\sE)/\sO_X \textup{id}\ar[r, "\overline{\sigma}"'] & T_X \ar[r]\ar[l, bend right, "\square"'] & 0
\end{tikzcd}
\end{center}
of the Atiyah exact sequence by $\sE\hspace{-0.07cm}\textit{nd}_{\,\sO_X}(\sE) \to \sE\hspace{-0.07cm}\textit{nd}_{\,\sO_X}(\sE)/\sO_X \textup{id}$. The connection is said to be \textit{flat} if the map $T_X \to \sE\hspace{-0.07cm}\textit{nd}_{\,\mathbb{C}}(\sE) / \sO_X \textup{id}$ is a morphism of Lie algebras. 

\medskip

A projective connection on $\sE$ is uniquely determined by the data of an open covering $(U_i)_{i\in I}$ of $X$ and connections $\square_i$ on $\sE|_{U_i}$ such that $\square_i=\square_j+\alpha_{ij}\otimes \textup{id}$ for some holomorphic $1$-form $\alpha_{ij}$ on $U_i\cap U_j$.

\medskip

Suppose that the connection $\square$ is flat. Then $\textup{K}(\square_i)=\beta_i \otimes \textup{id}$ for some holomorphic $2$-form $\beta_i$ on $U_i$, so that $\textup{tr}(\textup{K}(\square_i))=r \beta_i$. The holomorphic $2$-form $\textup{tr}(\textup{K}(\square_i))$ is the curvature of the connection on $\det \sE|_{U_i}$ induced by $\square_i$, and hence it is a closed $2$-form. Shrinking $U_i$ if necessary, we may therefore assume without loss of generality that $\textup{tr}(\textup{K}(\square_i))=\dd \alpha_i$ for some holomorphic $1$-form $\alpha_i$ on $U_i$.
Then 
\[
\textup{K}\left(\square_i-\frac{1}{r}\alpha_i\otimes \textup{id}\right)=\textup{K}(\square_i)-\frac{1}{r}\dd \alpha_i \otimes \textup{id}+\frac{1}{r^2}(\alpha_i\wedge \alpha_i)\otimes \textup{id}=0. 
\]
This shows that a flat projective connection on $\sE$ is uniquely determined by the data of flat connections $\square_i$ on $\sE|_{U_i}$ on an open covering $(U_i)_{i\in I}$ of $X$ such that $\square_i=\square_j+\alpha_{ij}\otimes \textup{id}$ for some closed holomorphic $1$-form $\alpha_{ij}$ on $U_i\cap U_j$. 

\medskip

We say that $\sE$ is \emph{projectively flat} if $\sE$ admits a flat projective connection.

\begin{fact}[{\cite[Corollary I.2.7]{kobayashi_diff_geom_vb}}]
The vector bundle $\sE$ is projectively flat if and only if the projective space bundle 
$\mathbb{P}(\sE):=\textup{Proj}_X(\textup{S}^\bullet\sE)$ comes from a representation $\pi_1(X) \to \textup{PGL}(r,\mathbb{C})$. 
\end{fact}

\begin{fact}[{\cite[Proposition II.3.1(b)]{kobayashi_diff_geom_vb}}]\label{fact:chern_classes}
If $\sE$ is projectively flat, then $c_i(\sE)=\frac{1}{r^i}\binom{r}{i}c_1(\sE)^i$ for any integer $i \ge 1$. 
\end{fact}

\begin{lemma}\label{lemma:proj_flat_frames}
The vector bundle $\sE$ is projectively flat if and only if there exists a covering $(U_i)_{i \in I}$ of $X$ by analytically open sets as well as frames $(e_1^i,\ldots,e_r^i)$ of $\sE$ over $U_i$ satisfying the following condition. There exist locally constant matrices $M_{ij} \in \textup{GL}(r,\sO_X(U_i))$ as well as nowhere vanishing holomorphic functions $f_{ij}$ on $U_i\cap U_j$ such that 
\[
\textbf{e}_i= f_{ij} M_{ij} \cdot \textbf{e}_j
\]
on $U_i \cap U_j$, where 
\[
\textbf{e}_i = 
\begin{pmatrix}
e_1^i \\
\vdots \\
e_r^i
\end{pmatrix}
.
\]
\end{lemma}

\begin{proof}
Suppose first that $\sE$ is projectively flat, and let $\square$ be a projective connection on $\sE$.
Let $(U_i)_{i \in I}$ be a covering of $X$ by analytically open sets, and let $\square_i$ be a flat connection on $\sE|_{U_i}$ lifting $\square|_{U_i}$. Shrinking the $U_i$, if necessary, we may assume that there exist frames
$(e_1^i,\ldots,e_r^i)$ of $\sE$ over $U_i$ such that $\square_i(\textbf{e}_i)=0$, where 
\[
\textbf{e}_i = 
\begin{pmatrix}
e_1^i \\
\vdots \\
e_r^i
\end{pmatrix}
.
\]
Let $\alpha_{ij}$ be closed holomorphic $1$-forms on $U_i\cap U_j$ such that $\square_i=\square_j+\alpha_{ij}\otimes \textup{id}$. Shrinking the $U_i$ further, we may assume that $\alpha_{ij}=\dd g_{ij}$ for some holomorphic function $g_{ij}$ on $U_i\cap U_j$. Let $N_{ij}\in \textup{GL}(r,\sO_X(U_i))$ be the transition matrix, so that $\textbf{e}_i=N_{ij} \cdot \textbf{e}_j$. We have
\begin{multline*}
0 = \square_i (\textbf{e}_i) = \square_i (N_{ij} \cdot \textbf{e}_j) = N_{ij} \cdot \square_i(\textbf{e}_j) + \dd N_{ij} \otimes \textbf{e}_j \\ = \alpha_{ij} \otimes \textbf{e}_i + \dd N_{ij}\cdot N_{ij}^{-1} \otimes \textbf{e}_i = \dd g_{ij} \otimes \textbf{e}_i + \dd N_{ij}\cdot N_{ij}^{-1} \otimes \textbf{e}_i.
\end{multline*}
This easily implies that $N_{ij}=\exp(-g_{ij}) M_{ij}$ for some locally constant matrix $M_{ij} \in \textup{GL}(r,\sO_X(U_i))$, so that $\textbf{e}_i= f_{ij} M_{ij} \cdot \textbf{e}_j$, where $f_{ij}=\exp(-g_{ij})$. 

Conversely, let $(U_i)_{i \in I}$ be a covering of $X$ by analytically open sets and let $(e_1^i,\ldots,e_r^i)$ be frames of $\sE$ over $U_i$ as in the statement of the lemma. Let $\square_i$ be the flat connection on $\sE|_{U_i}$ such that $\square_i (\textbf{e}_i)=0$. Then we have $\square_i=\square_j-\dd f_{ij}\otimes \textup{id}$ on $U_i \cap U_j$, and hence the $\square_i$ define a projective connection $\square$ on $\sE$.
\end{proof}

Let $f \colon X \to Y$ be a holomorphic map between complex manifolds, and let $\sE$ be a vector bundle on $Y$ with a projective connection $\square$ given by an open covering $(U_i)_{i\in I}$ of $X$ and connections $\square_i$ on $\sE|_{U_i}$ as above. The connections
$f|_{f^{-1}(U_i)}^*\square_i$ on $f^*\sE|_{f^{-1}(U_i)}$ then define a projective connection $f^*\square$ on $f^*\sE$. If $\square$ is flat, then so is $f^*\square$.

\medskip

Any flat vector bundle is projectively flat with vanishing Chern classes, so a natural question is the following.

\begin{question}
When is it true that a projectively flat vector bundle $\sE$ with $c_1(\sE)=0$ on a compact complex manifold is flat?
\end{question}

Numerically projectively flat vector bundles with vanishing first Chern class on complex projective manifolds are automatically numerically flat and hence flat. (We refer to \cite{druel_IMJ} and the references therein for theses notions.) 

The following lemma provides a partial result regarding the question posed above.

\begin{lemma}\label{lemma:projectively_flat_versus_flat}
Let $X$ be a complex projective manifold, and let $\sE$ be a projectively flat vector bundle on $X$ with $c_1(\sE)=0$. Then there exists a generically finite surjective morphism $f \colon Y \to X$ from a complex projective manifold $Y$ such that $f^*\sE$ is flat.  
\end{lemma}

\begin{proof}Let $r$ be the rank of $\sE$, and let 
\[
\rho\colon \pi_1(X) \to \textup{PGL}(r,\mathbb{C})=\textup{PSL}(r,\mathbb{C})
\]
be a representation that defines the projectively flat structure on $\sE$. By \cite[Lemma 2.6]{langer_simpson}, there exists a generically finite surjective morphism $f \colon Y \to X$ from a complex projective manifold $Y$ such that the composition
\[
\pi_1(Y) \to \pi_1(X) \to \textup{PSL}(r,\mathbb{C})
\]
lifts to a representation
\[
\pi_1(Y) \to \textup{SL}(r,\mathbb{C}).
\]
As a consequence, there exists a line bundle $\sL$ on $Y$ such that $(f^*\sE)\otimes\sL$ is a flat vector bundle. Now observe that $c_1(\sL)=0$. This is because Chern classes of flat vector bundles vanish by Fact \ref{fact:connection_vanishing_chern_classes} and $c_1(\sE)=0$ by assumption. This implies that $\sL$ is flat, finishing the proof of the lemma.
\end{proof}

We will need the following auxiliary statement.

\begin{lemma}\label{lemma:first_chern_class_direct_summand}
Let $X$ be a complex projective manifold, and let $\sE$ be a projectively flat vector bundle on $X$ with $c_1(\sE)=0$. Let $\sA$ be a direct summand of $\sE$. Then $c_1(\sA)=0$.  
\end{lemma}

\begin{proof}
By Lemma \ref{lemma:projectively_flat_versus_flat}, there exists a generically finite surjective morphism $f \colon Y \to X$ from a complex projective manifold $Y$ such that $f^*\sE$ is flat. Let now $C \to Y$ be a smooth projective curve. A theorem of Weil (\cite[Theorem 10]{atiyah57}) asserts that a vector bundle on a smooth projective curve is flat is and only if any direct summand has degree $0$. This immediately implies that $\deg(\sA|_C)=0$. The conclusion now follows from the projection formula.   
\end{proof}

\begin{question}\label{question:direct_summand_1}
When is it true that a direct summand of a flat vector bundle on a complex manifold is flat?
\end{question}

A theorem of Weil states that the answer to Question \ref{question:direct_summand_1} is yes if $\dim X=1$ (see \cite[Theorem 10]{atiyah57}).

\subsection{Ehresmann connections} 
Let $f \colon X \to Y$ be a surjective submersion between complex manifolds. An \emph{Ehresmann connection} on $f$ is a splitting of the exact sequence
\[
0 \to T_{X/Y}\to T_X \to f^*T_Y \to 0.
\]
The connection is said to be \emph{flat} if the corresponding subbundle $\sD \subseteq T_X$ is involutive.

\medskip

Suppose that the connection is flat, and let $\cF$ be the induced (regular) foliation on $X$. If $f$ is proper, then such a connection has the path lifting property, \emph{i.e.} given any point $y \in Y$, any path $\gamma\colon [0,1] \to Y$ starting at $y(0) = y$, and any point $x$ lying in the fibre $F=f^{-1}(y) \subseteq X$, there exists a unique path $\tilde{\gamma} \colon [0,1] \to X$ that is contained in a leaf of $\cF$ and starts at the point $\tilde{\gamma}(0)=x$. The monodromy representation at $y$ is the homomorphism 
\[
\pi_1(Y,y) \to \textup{Aut}(F)
\] 
defined by declaring that a homotopy class $[\gamma] \in \pi_1(Y,y)$ acts on $F$ by sending $x \in F$ to $\widetilde{\gamma^{-1}}(1)$ where $\widetilde{\gamma^{-1}}$ is the path lifting $\gamma^{-1}$ with initial condition $\widetilde{\gamma^{-1}}(0)=x$. Lifting arbitrary paths in Y based at y then gives a canonical isomorphism
\[
(\wt{Y} \times F)/\pi_1(Y,y) \cong X,
\]
where $\wt{Y}$ denotes the universal cover of $Y$ based at $y$ and $\pi_1(Y,y)$ acts diagonally on the product.

\section{Vector bundles on foliated manifolds}\label{section:connections_foliated_bundles}

\subsection{Partial connections}\label{paragraph:partial_connections}
Let $\cF$ be a regular foliation on a complex manifold $X$, and let $\dd_\cF \colon \Omega^{\bullet}_\cF \to \Omega^{\bullet + 1}_\cF$ be the exterior derivative along the leaves. Let $\sE$ be a vector bundle. A \textit{$\cF$-connection} (or a partial connection) $\nabla$ on $\sE$ is a morphism of abelian sheaves
\[
\nabla\colon \sE \to \Omega_\cF^1\otimes\sE
\] 
satisfying 
\[
\nabla(fe)=f\nabla(e)+\dd_\cF f \otimes e
\] 
for any local sections $f$ of $\sO_X$ and $e$ of $\sE$ over some open subset of $X$. The \emph{curvature} of $\nabla$ is the $\sO_X$-linear map
\[
\textup{K}(\nabla):=\nabla^1\circ\nabla\colon \sE \to \Omega_\cF^2\otimes \sE
\]
of $\sO_X$-modules, where $\nabla^i$ denotes the extension of $\nabla$ to 
a morphism of abelian sheaves $\Omega_\cF^i \otimes \sE \to \Omega_\cF^{i+1} \otimes \sE$ given by 
\[
\nabla^i(\alpha\otimes e)= \dd_\cF \alpha \otimes e + (-1)^i \alpha \wedge \nabla(e)
\]
for any local sections $\alpha$ of $\Omega_X^i$ and $e$ of $\sE$ over some open subset of $X$. 
The connection is said to be \textit{flat} if $\textup{K}(\nabla)=0$.

A connection $\nabla$ on $\sE$ can alternatively be defined as an $\sO_X$-linear map
\[
\nabla \colon T_\cF \longrightarrow  \sE\hspace{-0.07cm}\textit{nd}_{\,\mathbb{C}_X}(\sE)
\]
satisfying the Leibnitz rule 
\[
\nabla (v) (fe)=f\nabla (v) (e)+v(f) e
\] 
for any local sections $v$ of $T_\cF$, $f$ of $\sO_X$ and $e$ of $\sE$ over some open subset of $X$. 
The partial connection is flat if and only if the map $T_\cF \to \sE\hspace{-0.07cm}\textit{nd}_{\,\mathbb{C}_X}(\sE)$ is a morphism of Lie algebras. 

\medskip

Given partial connections induce new partial connections on associated vector bundles (see \ref{paragraph:connections}).

\medskip

Let $f \colon (X,\cF) \to (Y,\cG)$ be a \emph{morphism of foliated manifolds}, \textit{i.e.} the composition 
\[
f^*\Omega_Y^1 \longrightarrow \Omega_X^1 \longrightarrow \Omega_\cF^1
\] 
factors as 
\[
f^*\Omega_Y^1 \longrightarrow f^*\Omega_\cG^1 \overset{\eta}{\longrightarrow} \Omega_\cF^1,
\] 
and let $\sE$ be a vector bundle on $Y$ with a $\cG$-connection $\nabla$. The map 
\[
f^*\nabla \colon f^*\sE = \sO_X \otimes_{f^{-1}\sO_Y} f^{-1}\sE \to \Omega_\cF^1 \otimes f^*\sE 
\]
defined by 
\[
(f^*\nabla) (h \otimes f^{-1}e) = h ((\eta \otimes \textup{id})(f^*(\nabla (e)))) + \dd_\cF h \otimes f^{-1}e 
\]
for any local section $e$ of $\sE$ over some open subset $U$ of $Y$ and any holomorphic function $h$ on some open subset $V$ of $X$ such that $f(V) \subseteq U$ is well-defined and induces a $\cF$-connection $f^*\nabla$ on $f^*\sE$. If $\nabla$ is flat, then $f^*\nabla$ is flat as well. 

\begin{exmp}\label{exmp:canonical_connection_pull_back_foliation}
Let $f \colon X \to Y$ be a submersion of complex manifolds, and let $\cF$ be the regular foliation on $X$ induced by $f$. If $\sE$ is a vector bundle on $Y$, then $f^*\sE$ admits a natural flat $\cF$-connection. 
\end{exmp}

A vector bundle $\sE$ with a flat partial connection $\nabla$ is called a \emph{foliated vector bundle}. A \emph{morphism} between two foliated vector bundles $(\sE_1,\nabla_1)$ and $(\sE_2,\nabla_2)$ is an $\sO_X$-linear map $\phi \colon \sE_1 \to \sE_2$ such that
\[
\phi(\nabla_1(v)(e))=\nabla_2(v)(\phi(e))
\]
for any local sections $v$ and $e$ of $T_\cF$ and $\sE$ over some open subset of $X$.

\begin{rem}\label{remark:Ehresmann_versus_Koszul}
Let $(\sE,\nabla)$ be a foliated vector bundle on $X$, and let $E:=\mathbb{V}(\sE^\vee)$ with projection $\pi\colon E \to X$. Let $v$ and $\ell$ be any local sections of $T_\cF$ and $\sE^\vee$ over some open subset $U$ of $X$. The map sending $\ell$ to $v_E(\ell):=\nabla^\vee(v)(\ell)$ extends to a derivation
\[
v_E \colon \textup{S}^\bullet H^0(U,\sE^\vee|_U) \to \textup{S}^\bullet H^0(U,\sE^\vee|_U)
\]
by the Leibnitz rule, and gives a splitting of the exact sequence
\[
0 \to T_{E/X} \to T_E \to \pi^*T_X \to 0
\]
over $\pi^*T_\cF$. Let $\sD \subseteq T_X$ be the corresponding distribution. By construction, $\sD$ is $\textup{GL}(E/X)$-invariant, and involutive. The induced foliation $\cH$ on $X$ then projects onto $\cF$. In particular, $E$ is a foliated bundle in the sense of Molino.

Conversely, if $\cH$ a $\textup{GL}(E/X)$-invariant foliation on $E$ which projects on $\cF$, then 
$\cH$ comes from a flat partial connection on $\sE$.
\end{rem}

\begin{rem}
We refer to \cite{moerdijk_mrcun_foliations} for a standard reference concerning Lie groupoids.
Recall that the monodromy groupoid $\textup{G}=(G,X,s,t,c)$ of $\cF$ is the holomorphic Lie groupoid over $X$ for which the morphisms between two points in the same leaf $L$ are the homotopy classes of paths in $L$. If the points are not in the same leaf, then there is no morphism between them. The source and target maps pick out the endpoints of paths. The usual composition of paths then makes $\textup{G}$ into a Lie groupoid over $X$. 

A $\textup{G}$-equivariant vector bundle is a pair $(\sE,\psi)$, where $\sE$ a vector bundle on $X$, and $\psi\in \textup{Hom}_{\,\sO_G}(s^*\sE,t^*\sE)$) is an isomorphism satisfying the cocycle condition
\[
p_1^*\psi \circ p_2^* \psi = c^* \psi
\]
for $p_1,p_2$ the natural projections $G\times_{s,X,t} G \to G$.
A morphism between two $\textup{G}$-equivariant vector bundles $(\sE_1,\psi_1)$ and $(\sE_2,\psi_2)$ is an $\sO_X$-linear map $\phi \colon \sE_1 \to \sE_2$ such that 
\[
t^*\phi \circ \psi_1 = \psi_2 \circ s^*\phi\colon s^*\sE_1 \to t^*\sE_2.
\]

Let $(\sE,\nabla)$ be a foliated vector bundle on $X$, and let $L$ be a leaf of $\cF$. Then $\nabla|_L$ is a flat connection on $\sE|_L$. The parallel transport isomorphisms with respect to the flat connections $\nabla|_L$ give an equivalence between the category of foliated vector bundles and the category of $\textup{G}$-equivariant vector bundles. 
\end{rem}

\subsection{Bott partial connection}\label{exmp:bott_connection} 
Let $\cF$ be a regular foliation on a complex manifold $X$. Let $\sN$ be the normal bundle of $\cF$, and let $\pi\colon T_X\to \sN$ be the quotient map. For sections $e$ of $\sN$, $t$ of $T_X$, and $v$ of $\cF$ over some open subset of $X$ with $e=\pi(t)$, set 
\[
\nablaBott(v)(e)=\pi ([v,t]). 
\]
This expression is well defined, $\sO_X$-linear in $v$, and satisfies the Leibnitz rule 
\[\nablaBott(v)(fe)=f\nablaBott(v)(e)+v(f)e.
\]
The partial connection $\nablaBott$ is flat. The induced partial connection on $\sN^\vee$ is given by the formula
\[
\nablaBott^\vee(\alpha)=\cL_v \alpha =\iota_v \dd \alpha,
\]
for any local sections $\alpha$ of $\sN^\vee \subseteq \Omega_X^1$ and $v$ of $T_\cF$ over some open subset of $X$.
We will refer to $\nablaBott$ as the \textit{Bott partial connection} on $\sN$.

\medskip

The following lemma can be easily verified.

\begin{lemma}\label{lemma:functoriality_Bott_connection}
Let $f \colon (X,\cF) \to (Y,\cG)$ be a morphism of foliated manifolds. Then the natural map $\sN_\cF \to f^*\sN_\cG$ induces a morphism 
\[
(\sN_\cF,\nablaBott^\cF) \to (f^*\sN_\cG,f^*\nablaBott^\cG)
\]
of foliated vector bundles.
\end{lemma}

\subsection{Projective connections on foliated vector bundles}

Let $\cF$ be a regular foliation on a complex manifold $X$, and let $(\sE,\nabla)$ be a foliated vector bundle of rank $r$. 

\begin{defn}
$\, $
\begin{enumerate}
\item A \emph{connection} (resp. a \emph{flat connection}) on $(\sE,\nabla)$ is a connection (resp. a flat connection) $\square$ on $\sE$ such that $\nabla$ coincides with the composition 
\begin{center}
\begin{tikzcd}
\square_\cF \colon \sE \ar[r, "\square"] & \Omega_X^1\otimes\sE \ar[r, "{\textup{pr} \otimes \textup{id}}"] & \Omega_\cF^1 \otimes\sE.
\end{tikzcd}
\end{center}

We say that $(\sE,\nabla)$ is \emph{flat} if $(\sE,\nabla)$ admits a flat connection.

\item A \emph{projective connection} (resp. a \emph{flat projective connection}) on $(\sE,\nabla)$ is a projective connection (resp. a flat projective connection) $\square$ on $\sE$ such that the diagram  
\begin{center}
\begin{tikzcd}
& & \sD^1(\sE) \ar[d] \\
T_\cF \ar[r, hook] \ar[urr, bend left=15, "\nabla"] & T_X \ar[r,"\square"] & \sD^1(\sE)/\sO_X \textup{id}
\end{tikzcd}
\end{center}
commutes. 

We say that $(\sE,\nabla)$ is \emph{projectively flat} if $(\sE,\nabla)$ admits a flat projective connection. 
\end{enumerate}
\end{defn}

\begin{rem}\label{remark:projective_connection_foliated_vb}
Let $\square$ be a projective connection (resp. a flat projective connection) on $\sE$ given by 
an open covering $(U_i)_{i\in I}$ of $X$ and connections (resp. flat projective connections) $\square_i$ on $\sE|_{U_i}$ such that $\square_i|_{U_i\cap U_j}=\square_j|_{U_i\cap U_j}+\alpha_{ij}\otimes \textup{id}$ for some holomorphic $1$-form (resp. closed holomorphic $1$-form) $\alpha_{ij}$ on $U_i\cap U_j$. Then $\square$ is 
a projective connection (resp. a flat projective connection) on $(\sE,\nabla)$ if and only if 
$\nabla|_{U_i}= {\square_\cF}|_{U_i}+\beta_i \otimes \textup{id}$ for some $\beta_i\in H^0\left(U_i,\Omega_\cF^1|_{U_i}\right)$.
\end{rem}

\begin{rem}
If $(\sE,\nabla)$ is flat (resp. projectively flat), then $\sE$ is flat (resp. projectively flat). 
\end{rem}

\begin{lemma}\label{lemma:existence_horizontal_sections}
Let $x \in X$ be any point. Then there exists an open neighborhood $U$ of $x$ such that 
$(\sE|_U,\nabla|_U)$ is flat.
\end{lemma}

\begin{proof}
Let $E:=\mathbb{V}(\sE^\vee)$ with projection $\pi\colon E \to X$, and let $\cH$ be the foliation on $E$ induced by $\nabla$ (see Remark \ref{remark:Ehresmann_versus_Koszul}). Let $e$ be any local section of $\sE$ over some open subset $U$ of $x$, viewed as a section of $\pi$ over $U$. A simple computation shows that $\nabla(e)=0$ if and only if $e(U)$ is invariant under $\cH|_{\pi^{-1}(U)}$. It then follows from \cite[Proposition 2.7]{molino} that there exists an open neighborhood $U$ of $x$ and a frame $(e_1,\ldots,e_r)$ of $\sE$ over $U$ such that $\nabla(e_i)=0$ for any index $i$. The (flat) connection $\square$ on $\sE|_U$ such that $\square(e_i)=0$ for any index $i$ gives a flat connection $\square$ on $(\sE|_U,\nabla|_U)$. This finishes the proof of the lemma.
\end{proof}

\begin{lemma}\label{lemma:trans_proj_flat_frames}
The foliated vector bundle $(\sE,\nabla)$ is projectively flat if and only if there exists a covering $(U_i)_{i \in I}$ of $X$ by analytically open sets as well as frames $(e_1^i,\ldots,e_r^i)$ of $\sE$ over $U_i$ satisfying the following properties. 
\begin{enumerate}
\item There exists $\beta_i\in H^0\left(U_i,\Omega_\cF^1|_{U_i}\right)$ such that
\[
\nabla(\textbf{e}_i)=\beta_i\otimes \textbf{e}_i,
\]
where 
\[
\textbf{e}_i = 
\begin{pmatrix}
e_1^i \\
\vdots \\
e_r^i
\end{pmatrix}
.
\]
\item There exist locally constant matrices $M_{ij} \in \textup{GL}(r,\sO_X(U_i))$ as well as nowhere vanishing holomorphic functions $f_{ij}$ on $U_i\cap U_j$ such that 
\[
\textbf{e}_i= f_{ij} M_{ij} \cdot \textbf{e}_j
\]
on $U_i \cap U_j$.
\end{enumerate}
\end{lemma}

\begin{proof}The claim follows easily from Lemma \ref{lemma:proj_flat_frames} and Remark \ref{remark:projective_connection_foliated_vb} above.
\end{proof}

The proof of Lemma \ref{lemma:trans_flat_frames} below is very similar to that of Lemma \ref{lemma:trans_proj_flat_frames}. We leave it to the reader to spell out the details.

\begin{lemma}\label{lemma:trans_flat_frames}
The foliated vector bundle $(\sE,\nabla)$ is flat if and only if there exists a covering $(U_i)_{i \in I}$ of $X$ by analytically open sets as well as frames $(e_1^i,\ldots,e_r^i)$ of $\sE$ over $U_i$ satisfying the following properties. 
\begin{enumerate}
\item We have
\[
\nabla(\textbf{e}_i)=0,
\]
where 
\[
\textbf{e}_i = 
\begin{pmatrix}
e_1^i \\
\vdots \\
e_r^i
\end{pmatrix}
.
\]
\item There exist locally constant matrices $M_{ij} \in \textup{GL}(r,\sO_X(U_i))$ such that 
\[
\textbf{e}_i= M_{ij} \cdot \textbf{e}_j
\]
on $U_i\cap U_j$.
\end{enumerate}
\end{lemma}

\begin{prop}\label{prop:projectively_flat_versus_flat}
The following statements hold. 
\begin{enumerate}
\item The foliated vector bundle $(\sE,\nabla)$ admits a projective connection if and only if $(\sE\otimes\sE^\vee,\nabla\otimes\nabla^\vee)$ admits a connection.
\item The foliated vector bundle $(\sE,\nabla)$ is projectively flat if and only if $(\sE\otimes\sE^\vee,\nabla\otimes\nabla^\vee)$ is flat.
\end{enumerate}
\end{prop}

\begin{proof}
Let $\square$ be a projective connection on $(\sE,\nabla)$ given by an open covering $(U_i)_{i\in I}$ of $X$ and connections $\square_i$ on $\sE|_{U_i}$ such that $\square_i=\square_j+\alpha_{ij}\otimes \textup{id}$ for some holomorphic $1$-form $\alpha_{ij}$ on $U_i\cap U_j$. Then the connections 
$\square_i \otimes \square^\vee_i$ on $\sE\otimes\sE^\vee|_{U_i}$ satisfy
\begin{align*}
\square_i \otimes \square^\vee_i & = \square_j \otimes \square^\vee_j + [\alpha_{ij}\otimes \textup{id}, \bullet\, ] \\
& =  \square_j \otimes \square^\vee_j 
\end{align*}
on $U_i \cap U_j$. One then readily checks that the $\square_i \otimes \square^\vee_i$ define a connection on $(\sE\otimes\sE^\vee,\nabla\otimes\nabla^\vee)$. Moreover, a standard calculation shows that $\textup{K}(\square_i \otimes \square^\vee_i) = 0$ if $\textup{K}(\square_i)=0$. Thus, if $(\sE,\nabla)$ is projectively flat, then $(\sE\otimes\sE^\vee,\nabla\otimes\nabla^\vee)$ is flat.

The proof of the converse statement is similar to that of \cite[Proposition 1.2]{biswas_semistability}, and so we leave some easy details to the reader. The first step of the argument of Biswas applies to show that the $\sO_X$-linear map
\[
\sE\hspace{-0.07cm}\textit{nd}_{\,\sO_X}(\sE)/ \sO_X\textup{id} \to  \sE\hspace{-0.07cm}\textit{nd}_{\,\sO_X}(\sE\hspace{-0.07cm}\textit{nd}_{\,\sO_X}(\sE))  
\]
defined by 
\[
\bar{u} \mapsto [u,\bullet\,]
\]
admits a section. Let 
\[
\Lambda \colon \Omega_X^1 \otimes \sE\hspace{-0.07cm}\textit{nd}_{\,\sO_X}(\sE\hspace{-0.07cm}\textit{nd}_{\,\sO_X}(\sE)) \to \Omega_X^1 \otimes (\sE\hspace{-0.07cm}\textit{nd}_{\,\sO_X}(\sE)/ \sO_X\textup{id})
\]
be the induced map. 

Let now $\textup{D}$ be any connection on $(\sE\otimes\sE^\vee,\nabla\otimes\nabla^\vee)$. Let $(U_i)_{i\in I}$ be an open covering of $X$, and let $\square_i$ be flat connections on $(\sE|_{U_i},\nabla|_{U_i})$ (see Lemma \ref{lemma:existence_horizontal_sections}). Let also
\[
\textup{A}_i \in H^0(U_i,\Omega_{U_i}^1 \otimes \sE\hspace{-0.07cm}\textit{nd}_{\,\sO_{U_i}}(\sE\hspace{-0.07cm}\textit{nd}_{\,\sO_{U_i}}(\sE|_{U_i})))
\] 
such that 
\[
\textup{D}|_{U_i}=\square_i \otimes \square^\vee_i + \textup{A}_i.
\]
Shrinking the $U_i$ further, we may assume that there exists
\[
\textup{a}_i \in H^0(U_i,\Omega_{U_i}^1 \otimes \sE\hspace{-0.07cm}\textit{nd}_{\,\sO_{U_i}}(\sE|_{U_i}))
\]
such that 
\[
\Lambda(\textup{A}_i)=\bar{\textup{a}}_i \in H^0(U_i,\Omega_{U_i}^1 \otimes (\sE\hspace{-0.07cm}\textit{nd}_{\,\sO_{U_i}}(\sE|_{U_i})/\sO_{U_i}\textup{id})).
\]
Note that both $D|_{U_i}$ and $\square_i \otimes \square^\vee_i$ are connections on $(\sE\otimes\sE^\vee|_{U_i},\nabla\otimes\nabla^\vee|_{U_i})$, so that
\[
\textup{A}_i \in H^0(U_i,\sN^\vee|_{U_i} \otimes \sE\hspace{-0.07cm}\textit{nd}_{\,\sO_{U_i}}(\sE\hspace{-0.07cm}\textit{nd}_{\,\sO_{U_i}}(\sE|_{U_i}))) \subseteq H^0(U_i,\Omega_{U_i}^1 \otimes \sE\hspace{-0.07cm}\textit{nd}_{\,\sO_{U_i}}(\sE\hspace{-0.07cm}\textit{nd}_{\,\sO_{U_i}}(\sE|_{U_i}))).
\]
As a consequence, we may also assume without loss of generality that 
\[
\textup{a}_i \in H^0(U_i,\sN^\vee|_{U_i} \otimes \sE\hspace{-0.07cm}\textit{nd}_{\,\sO_{U_i}}(\sE|_{U_i})) \subseteq H^0(U_i,\Omega_{U_i}^1 \otimes \sE\hspace{-0.07cm}\textit{nd}_{\,\sO_{U_i}}(\sE|_{U_i})).
\]

Let $\overline{\square}_i:=\square_i+\textup{a}_i$. By choice of $\textup{a}_i$, $\overline{\square}_i$ is a connection on $(\sE|_{U_i},\nabla|_{U_i})$. Moreover,
\[
\overline{\square}_i \otimes \overline{\square}_i^\vee = \square_i \otimes \square^\vee_i + [\textup{a}_i,\bullet\,],
\]
so that 
\[
\overline{\square}_i \otimes \overline{\square}_i^\vee - \overline{\square}_j \otimes \overline{\square}_j^\vee= \square_i \otimes \square^\vee_i - \square_j \otimes \square^\vee_j + [\textup{a}_i-\textup{a}_j,\bullet\,]=\textup{A}_j-\textup{A}_i + [\textup{a}_i-\textup{a}_j,\bullet\,]
\]
on $U_i\cap U_j$, and hence
\[
\Lambda ( \overline{\square}_i \otimes \overline{\square}_i^\vee - \overline{\square}_j \otimes \overline{\square}_j^\vee ) = 0.
\]
On the other hand, we have
\[
\overline{\square}_i \otimes \overline{\square}_i^\vee - \overline{\square}_j \otimes \overline{\square}_j^\vee= \left[\overline{\square}_i - \overline{\square}_j,\bullet\,\right],
\]
hence there exist $1$-forms $\alpha_{ij}$ on $U_i\cap U_j$ such that $\overline{\square}_i = \overline{\square}_j + \alpha_{ij} \otimes \textup{id}$. The $\overline{\square}_i$ then define a projective connection $\square$ on $(\sE,\nabla)$. Finally, one readily checks that if $\textup{D}$ is flat, then $\square$ is flat as well. This completes the proof of the proposition.
\end{proof}

\begin{lemma}\label{lemma:functoriality_flatness}
Let $f \colon (X,\cF) \to (Y,\cG)$ be a morphism of foliated manifolds, and let $(\sE,\nabla)$ be a foliated vector bundle on $Y$. 
\begin{enumerate}
\item If $\square$ is a connection $($resp. a projective connection$)$ on $(\sE,\nabla)$, then $f^*\square$ is a connection $($resp. a projective connection$)$ on $(f^*\sE,f^*\nabla)$. 
\item If $(\sE,\nabla)$ is flat $($resp. projectively flat$)$, then $(f^*\sE,f^*\nabla)$ is flat $($resp. projectively flat$)$ as well.
\end{enumerate}
\end{lemma}

\begin{proof}
Let $\square$ be a connection on $(\sE,\nabla)$. By definition, $\nabla$ coincides with the composition 
\begin{center}
\begin{tikzcd}
\sE \ar[r, "\square"] & \Omega_Y^1\otimes\sE \ar[r, "{\textup{pr} \otimes \textup{id}}"] & \Omega_\cG^1 \otimes\sE.
\end{tikzcd}
\end{center}
The lemma then follows easily from the very definitions of $f^*\nabla$ and $f^*\square$.
\end{proof}

The following Lemma treat the converse statements in a special setting.

\begin{lemma}\label{lemma:descent}
Let $f \colon X \to Y$ be a surjective submersion with connected fibers between complex manifolds, and let $\cF$ be the regular foliation on $X$ induced by $f$. Suppose that $f$ is proper. Let $\sE$ be a vector bundle on $Y$, and let $\nabla$ be the flat $\cF$-connection on $\sG:=f^*\sE$ as in Example \ref{exmp:canonical_connection_pull_back_foliation}. 
\begin{enumerate}
\item If $(\sG,\nabla)$ admits a connection $($resp. a flat connection$)$, then $\sE$ admits a connection $($resp. a flat connection$)$. 
\item If $(\sG,\nabla)$ admits a projective connection $($resp. a flat projective connection$)$, then $\sE$ admits a projective connection $($resp. a flat projective connection$)$. 
\end{enumerate}
\end{lemma}

\begin{proof}
The second statement follows from the first one together with Proposition \ref{prop:projectively_flat_versus_flat}. Thus, we concentrate on the first assertion. 

Let $\square$ be connection on $(\sG,\nabla)$. Let $(e_1,\ldots,e_r)$ be a local from of $\sE$ on some open subset $U$ of $Y$ with $r$ the rank of $\sE$, and let $V:=f^{-1}(U)$. Let $\dd_V$ be the flat connection on $\sG|_V$ such that $\dd_V (f^*e_i)=0$ for any index $i$. Let also $\textup{A}_V\in H^0(V,\Omega^1_V \otimes \sE\hspace{-0.07cm}\textit{nd}_{\,\sO_V}(\sG|_V))$ such that $\square|_V=\dd_V + \textup{A}_V$. Note that both $\square|_V$ and $\dd_V$ are connections on $(\sG|_V,\nabla|_V)$, so that
$\textup{A}_V\in H^0(V,(f|_V)^*(\Omega^1_U \otimes \sE\hspace{-0.07cm}\textit{nd}_{\,\sO_U}(\sE|_U))) \subseteq H^0(V,\Omega^1_V \otimes \sE\hspace{-0.07cm}\textit{nd}_{\,\sO_V}(\sG|_V))$. This in turn implies that there exists $\textup{B}_U\in H^0(U,\Omega^1_U \otimes \sE\hspace{-0.07cm}\textit{nd}_{\,\sO_U}(\sE|_U))$ such that $\textup{A}_V=f^*\textup{B}_U$ since $f$ is proper with connected fibers. Let $\dd_U$ be the connection on $\sE_U$ such that $\dd_U e_i = 0$ for any index $i$. The connection $\textup{D}_U:=\dd_U+\textup{B}_U$ on $\sE|_U$ then satisfies $\square|_V = (f|_V)^*\textup{D}_U$. Note that $\textup{D}_U$ is uniquely determined by this property. As a consequence, there is a connection $\textup{D}$ on $\sE$ such that $\square=f^*\textup{D}$. Local calculations show that $\square$ is flat if and only if $\textup{D}$ is flat. This finishes the proof of the lemma.
\end{proof}

\section{Transversely projectively flat foliations: definition, examples and a vanishing theorem for Chern classes}\label{section:definition}

The key notion is that of a (transversely) projectively flat foliation.

\begin{defn}
Let $\cF$ be regular foliation with normal bundle $\sN$ on a complex connected manifold $X$. We say that $\cF$ is transversely flat (resp. transversely projectively flat) if the foliated vector bundle $(\sN,\nabla^\textup{B})$ is flat (resp. projectively flat).
\end{defn}

\begin{exmp}\label{example:linear_foliation}
Let $X$ be a complex manifold, and let $\omega_1,\ldots,\omega_q \in H^0(X,\Omega^1_X)$ be pointwise linearly independent closed holomorphic $1$-forms. The annihilator $T_\cF \subseteq T_X$ of $\sN^\vee=\omega_1\sO_X \oplus \cdots \oplus \omega_q \sO_X \subseteq \Omega_X^1$ thus defines a regular foliation $\cF$ of codimension $q$. In addition, we have $\nablaBott^\vee(\omega_i)=0$. The connection $\square$ on $\sN^\vee$ satisfying $\square(\omega_i)=0$ for any index $1 \le i \le q$ is a flat connection on $(\sN^\vee,\nablaBott^\vee)$. In other words, $\cF$ is transversely flat.

\medskip

In particular, linear foliations on complex tori are transversely flat.
\end{exmp}

\begin{exmp}
Let $T$ be a complex tori, and let $\alpha$ be a nonzero holomorphic $1$-form on $T$. Le $Y$ be a complex manifold, and let $\beta$ be nonzero closed $1$-form on $Y$. Then $\textup{pr}_T^* \alpha+ \textup{pr}_Y^*\beta$ defines a transversely flat foliation on $T \times Y$ by Example \ref{example:linear_foliation}.
\end{exmp}

\begin{exmp}
Let $\sL_1,\ldots,\sL_q$ be flat line bundles on a compact Kähler manifold $X$, and let $\omega_i \in H^0(X,\Omega^1_X\otimes \sL_i)$ be pointwise linearly independent holomorphic twisted $1$-forms. Then $\nabla^1(\omega_i)=0 \in H^0(X,\Omega_X^2\otimes \sL_i)$ by Hodge theory for unitary local systems. 
As in Example \ref{example:linear_foliation} above, the annihilator $T_\cF \subseteq T_X$ of $\sN^\vee=\sL_1^\vee \oplus \cdots \oplus \sL_q^\vee \subseteq \Omega_X^1$ then defines a transversely flat foliation $\cF$ of codimension $q$.
\end{exmp}

\begin{exmp}\label{example:submersion}
Let $\cG$ be a regular foliation on a complex manifold $Y$, and let $f \colon X \to Y$ be a submersion. Let also $\cF:=f^{-1}\cG$. We have $(\sN_\cF,\nablaBott^\cF) \cong (f^*\sN_\cG,f^*\nablaBott^\cG)$ by Lemma \ref{lemma:functoriality_Bott_connection}. If $\cG$ is transversely flat (resp. transversely projectively flat), then Lemma \ref{lemma:functoriality_flatness} applies to show that $\cF$ is transversely flat (resp. transversely projectively flat) as well. 
\end{exmp}

\begin{exmp}\label{example:blow_up}
Let $\cG$ be regular foliation of codimension $q$ on a complex manifold $Y$, and let $Z \subseteq Y$ be a closed submanifold of dimension at least $q$. Suppose in addition that $Z$ is everywhere transverse to $\cG$, and let $\beta\colon X \to Y$ be the blow-up of $Y$ along $Z$. Let $f \colon U \to \mathbb{C}^q$ be a submersion defining $\cG$ on some open subset $U \subseteq Y$, and let also $V:=\beta^{-1}(U)$. By our assumptions, the composition $f \circ (\beta|_V) \colon V \to \mathbb{C}^q$ is a submersion as well, so that $\cG$ induces a regular foliation $\cF$ on $X$ with $\sN_\cF \cong \beta^*\sN_\cG$. 

By Lemma \ref{lemma:functoriality_Bott_connection} together with Lemma \ref{lemma:functoriality_flatness} again, if $\cG$ is transversely flat (resp. transversely projectively flat), then so is $\cF$.
\end{exmp}

\begin{exmp}\label{example:conic_bundle}
Let $f \colon X \to Y$ be a conic bundle, where $X$ and $Y$ are complex manifolds. Suppose that the discriminant locus 
\[
\Delta:=\{y \in Y \,|\, f^{-1}(y) \textup{ is a singular conic}\}
\] 
is smooth, or equivalently, that any fiber of $f$ is a (possibly singular) reduced conic (see \cite[Lemma 4.7]{druel_bbcd2}). Let $\cG$ be a regular foliation of dimension $p$ at least $2$ on $Y$. Suppose that $\Delta$ is everywhere transverse to $\cG$. The holomorphic map $f$ factors locally in $Y$ as the blow-up of a smooth conic bundle along of codimension $2$ submanifold. By Examples \ref{example:submersion} and \ref{example:blow_up}, $\cG$ induces a regular foliation $\cF$ on $X$ with $\sN_\cF \cong \beta^*\sN_\cG$, and hence $\cF$ is transversely flat (resp. transversely projectively flat) if $\cG$ is. 
\end{exmp}

The following is the main result of this section. Note that Theorem \ref{thm_intro:vanishing} is an immediate consequence of Theorem \ref{thm:vanishing} below.

\begin{thm}\label{thm:vanishing}
Let $X$ be a complex manifold, and let $\cF$ be regular foliation of codimension $q \ge 2$ with normal bundle $\sN$. If $\cF$ is transversely projectively flat, then $(\sN,\nablaBott)$ admits a connection. In particular, we have $c_i(\sN)=0 \in H^i(X,\Omega_X^i)$ for any index $i \ge 1$.
\end{thm}

Before proving the theorem below, we note the following corollary which is a special case of Theorem \ref{thm:vanishing}.

\begin{cor}
Let $X$ be a complex manifold. If $T_X$ is projectively flat, then $T_X$ admits a holomorphic connection.
\end{cor}

\begin{proof}[Proof of Theorem \ref{thm:vanishing}] By Lemma \ref{lemma:trans_proj_flat_frames}, there exists a covering $(U_i)_{i \in I}$ of $X$ by analytically open sets as well as frames $(\alpha_1^i,\ldots,\alpha_q^i)$ of $\sN^\vee\subseteq \Omega_X^1$ over $U_i$ satisfying the following properties. There exists $\beta_i\in H^0(U_i,\Omega^1_{\sF}|_{U_i})$ such that
\[
\nablaBott^\vee(\textbf{a}_i)=\beta_i\otimes \textbf{a}_i,
\]
where 
\[
\textbf{a}_i = 
\begin{pmatrix}
\alpha_1^i \\
\vdots \\
\alpha_q^i
\end{pmatrix}
.
\]
Moreover, there exist locally constant matrices $M_{ij} \in \textup{GL}(q,\sO_X(U_i))$ as well as nowhere vanishing holomorphic functions $f_{ij}$ on $U_{ij}:=U_i\cap U_j$ such that 
\[
\textbf{a}_i= f_{ij} M_{ij} \cdot \textbf{a}_j
\]
on $U_{ij}$. We have
\begin{align*}
\beta_i\otimes \textbf{a}_i & =  \nablaBott^\vee(\textbf{a}_i) \\
& = \nablaBott^\vee(f_{ij} M_{ij} \cdot \textbf{a}_j)\\
& =  f_{ij} M_{ij} \cdot  \nablaBott^\vee(\textbf{a}_j) + \dd_\cF f_{ij} \otimes (M_{ij} \cdot \textbf{a}_j) \\
& =  f_{ij} \beta_j \otimes M_{ij} \cdot \textbf{a}_j + \frac{\dd_\cF f_{ij}}{f_{ij}} \otimes \textbf{a}_i \\
& =  \beta_j \otimes \textbf{a}_i + \frac{\dd_\cF f_{ij}}{f_{ij}} \otimes \textbf{a}_i 
\end{align*}
on $U_{ij}$, or, equivalently,
\begin{equation}\label{eq:beta}
{\frac{\dd f_{ij}}{f_{ij}}}\bigg|_{T_{\cF}|_{U_{ij}}}=(\beta_i-\beta_j)|_{U_{ij}}.
\end{equation}

Now, set 
\[
A_k^i=\alpha_1^i \wedge \cdots \wedge \widehat{\alpha_k^i} \wedge \cdots \wedge \alpha_q^i \in H^0(U_i,\wedge^{q-1} \sN^\vee|_{U_i}) \subseteq H^0(U_i,\Omega_{U_i}^{q-1}).
\]
and 
\[
\textbf{A}_i = 
\begin{pmatrix}
A_1^i \\
\vdots \\
A_q^i
\end{pmatrix}
.
\]
Then, we have
\[
\nablaBott^\vee(\textbf{A}_i) = (q-1)\beta_i \otimes \textbf{A}_i,
\] 
where, by abuse of notation, we use the same symbol $\nablaBott^\vee$ to denote the connection on $\wedge^{q-1}\sN^\vee$ induced by $\nablaBott^\vee$. This is because
\begin{multline*}
\nablaBott^\vee(A_k^i)=\nablaBott^\vee(\alpha_1^i) \wedge \cdots \wedge \widehat{\alpha_k^i} \wedge \cdots \wedge \alpha_q^i+\cdots \\ +\alpha_1^i \wedge \cdots \wedge \widehat{\nablaBott^\vee(\alpha_k^i)} \wedge \cdots \wedge \alpha_q^i+\cdots +
\alpha_1^i \wedge \cdots \wedge \widehat{\alpha_k^i} \wedge \cdots \wedge \nablaBott^\vee(\alpha_q^i) = (q-1) \beta_i \otimes A_k^i.
\end{multline*}

On the other hand, there exist $1$-forms $\gamma_{kl}^i \in H^0\big(U_i,\Omega_{U_i}^1\big)$ such that 
\[
\dd\alpha_k^i = \sum_{l=1}^{q} \gamma_{kl}^i \wedge \alpha_l^i
\]
since $T_\cF$ is involutive. This immediately implies that there exist $1$-forms $\eta_{kl}^i \in H^0\big(U_i,\Omega_{U_i}^1\big)$ such that
\[
\dd A_k^i = \sum_{l=1}^{q} \eta_{kl}^i \wedge A_l^i,
\]
so that
\[
\nablaBott^\vee(A_k^i) = \sum_{l=1}^{q} \eta_{kl}^i|_{T_{\cF}|_{U_i}} \otimes A_l^i
\]
since $\nablaBott^\vee(v)(A_k^i)=\cL_v(A_k^i)=\iota_v \dd A_k^i$ for any section $v$ of $T_\cF$ over $U_i$
by the very definition of $\nablaBott^\vee$ (see \ref{exmp:bott_connection}). This in turn implies $\eta_{kl}^i|_{T_{\cF}|_{U_i}}=0$ if $l \neq k$ and $\eta_{kk}^i|_{T_{\cF}|_{U_i}}=(q-1)\beta_i$. In other words, $\eta_{kl}^i \in H^0(U_i,\sN^\vee|_{U_i})$ if $l \neq k$. We may therefore assume without loss of generality
\[
\eta_{kl}^i=0 \textup{ if } l\neq k
\] 
since 
\[
(-1)^{l} \alpha_l^i\wedge A_l^i = - \alpha_1^i\wedge\cdots \wedge \alpha_q^i= (-1)^{m} \alpha_m^i\wedge A_m^i
\]
for any indices $l$ and $m$, and $\alpha_m^i\wedge A_l^i=0$ if $l\neq m$. Note that $\eta_{kk}^i-\eta_{11}^i \in H^0(U_i,\sN^\vee|_{U_i})$ since $\eta_{ll}^i|_{T_{\cF}|_{U_i}}=(q-1)\beta_i$ for any index $l$. Thus, there exists a holomorphic function $c_i$ on $U_i$ such that 
\[
(\eta_{kk}^i-\eta_{11}^i)\wedge A_k^i=c_i \alpha_k^i \wedge A_k^i.
\]
Set 
\[
\eta_i=\eta_{11}^i+\sum_{l=1}^q c_l \alpha_l^i. 
\]
Then, we have
\begin{equation}\label{eq:eta}
\eta_i|_{T_{\cF}|_{U_i}}=\eta_{11}^i|_{T_{\cF}|_{U_i}}=(q-1)\beta_i
\end{equation}
and
\[
\eta_i \wedge A_k^i = \eta_{11}^i \wedge A_k^i + \sum_{l=1}^q c_l \alpha_l^i \wedge A_k^i = \eta_{11}^i \wedge A_k^i + c_k\alpha_k^i \wedge A_k^i = \eta_{kk}^i \wedge A_k^i = \dd A_k^i.
\]
By choice of the $A_k^i$, there exists a locally constant matrix $N_{ij} \in \textup{GL}(q,\sO_X(U_i))$ such that 
\[
\textbf{A}_i= f_{ij}^{q-1} N_{ij} \cdot \textbf{A}_j.
\]
Therefore, we have
\begin{align*}
\eta_i \wedge \textbf{A}_i & =  \dd \textbf{A}_i \\
& = \dd (f_{ij}^{q-1} N_{ij} \cdot \textbf{A}_j) \\
& =  f_{ij}^{q-1} N_{ij} \cdot \dd \textbf{A}_j + (q-1) f_{ij}^{q-2} \dd f_{ij} \wedge (N_{i,j} \cdot \textbf{A}_j) \\
& =  f_{ij}^{q-1} \eta_j \wedge N_{ij} \cdot \textbf{A}_j + (q-1)\frac{\dd f_{ij}}{f_{ij}} \wedge \textbf{A}_i\\
& =  \eta_j \wedge \textbf{A}_i + (q-1)\frac{\dd f_{ij}}{f_{ij}} \wedge \textbf{A}_i
\end{align*}
on $U_{ij}$, or, equivalently,
\begin{equation}\label{eq:almost_coboundary}
\left((q-1)\frac{\dd f_{ij}}{f_{ij}} +\eta_j-\eta_i\right)\wedge A_k^i=0
\end{equation}
for any index $k$. Note that
\[
\left((q-1)\frac{\dd f_{ij}}{f_{ij}} +\eta_j-\eta_i\right)\bigg|_{T_\cF|_{U_{ij}}}=
(q-1)\left(\frac{\dd f_{ij}}{f_{ij}} +\beta_j-\beta_i\right)\bigg|_{T_\cF|_{U_{ij}}}=0
\]
by \eqref{eq:beta} and \eqref{eq:eta}. This implies 
\[
(q-1)\frac{\dd f_{ij}}{f_{ij}} = (\eta_i - \eta_j)|_{U_{ij}}.
\]
Note that $q-1 > 0$ by assumption, and let $\nabla_i^\vee$ be the connection on $\sN^\vee|_{U_i}$ given by $\dd_i+\frac{1}{q-1}\eta_i\otimes \textup{id}$ with $d_i$ the flat connection on $\sN^\vee|_{U_i}$ such that $\dd_i \alpha_k^i=0$ for any index $k$. Then, we have
\begin{align*}
\nabla_j^\vee(\textbf{a}_i) & =\nabla_j^\vee(f_{ij} M_{ij} \cdot \textbf{a}_j)\\
& =f_{ij} M_{ij} \otimes \nabla_j^\vee(\textbf{a}_j)+\dd f_{ij} \otimes M_{ij}\cdot \textbf{a}_j \\ 
& = \frac{1}{q-1} \eta_j\otimes f_{ij} M_{ij} \cdot \textbf{a}_j+\frac{\dd f_{ij}}{f_{ij}}\otimes f_{ij} M_{ij}\cdot \textbf{a}_j\\
& =\frac{1}{q-1} \eta_j \otimes \textbf{a}_i+\frac{1}{q-1}(\eta_i-\eta_j)\otimes \textbf{a}_i=\frac{1}{q-1}\eta_i\otimes \textbf{a}_i\\
& =\nabla_i^\vee(\textbf{a}_i),
\end{align*}
and hence the $\nabla_i^\vee$ glue together defining a connection $\nabla^\vee$ on $\sN^\vee$. In addition, we have $\nabla^\vee|_{T_\cF|_{U_i}} = \nabla_i^\vee|_{T_\cF|_{U_i}} = \nablaBott^\vee|_{U_i}$ by \eqref{eq:eta}, and hence $\nabla:=(\nabla^\vee)^\vee$ is a connection on $(\sN,\nablaBott)$. This finishes the proof of the theorem.
\end{proof}

\begin{rem}
If one drops the assumption that the projective connection on $\sN$ is compatible with the Bott connection, then the conclusion of Theorem \ref{thm:vanishing} may be false. Indeed, let $E$ be an elliptic curve, and let $B$ be an abelian variety of dimension $\dim B =n-1 \ge  1$. Set $A:= E \times B$. Let $\sM$ be a line bundle of degree $2$ on $E$, and set $\sL:=\textup{pr}_E^*\sM^\vee$. Let $s_i, t_i \in H^0(A,\sL^\vee)$ for $i \in \{1,\ldots,n-1\}$ such that the zero sets of the $s_i$ and the $t_i$ are pairwise disjoint. Let $\alpha_1,\ldots,\alpha_n$ be a basis of $H^0(A,\Omega_A^1)$.
The sections $s_i$ and $t_i$ then give an injective map of vector bundles $\sL \subset \sO_A \alpha_i \oplus \sO_A \alpha_{i+1}$. In addition, the induced map of vector bundles $\sN^\vee:=\sL^{\oplus n-1} \to \Omega_A^1$ is injective as well. The annihilator $T_\cF \subset T_X$ of $\sN^\vee \subset \Omega_X^1$ defines a regular foliation $\cF$ of dimension $1$ on $A$ whose normal bundle is projectively flat with non-vanishing first Chern class by construction.
\end{rem}

\section{Transversely projectively flat foliations with compact leaves}\label{section:compact_leaves}

The proof of Theorem \ref{thm:compact_leaves} below makes use of the following characterization of torus quotients. Note that Theorem \ref{thm_intro:projectively_flat_spaces} is an immediate consequence of Theorem \ref{thm:projectively_flat_spaces} below.

\begin{thm}\label{thm:projectively_flat_spaces}
Let X be a compact K\"{a}hler variety of dimension $n$ at least $2$ with klt singularities, and let $U \subseteq X_{\textup{reg}}$ be a Zariski open subset with complement of codimension at least $2$. Suppose that $T_U$ is projectively flat. Then there exists a complex torus $T$, and a quasi-\'etale cover $T \to X$.
\end{thm}

\begin{proof}
The arguments used in the proof of \cite[Theorem 5.1]{CGGN22} show that $X$ admits a quasi-\'etale cover $\eta \colon Y \to X$ satisfying the following properties. There exists an open subset $X^\circ \subseteq X$ whose complement is a union of finitely many points such that $Y^\circ:=\eta^{-1}(X^\circ) \to X^\circ$ is a maximally quasi-\'etale cover with $Y^\circ$ smooth (we refer the reader to \cite[Definition 5.3]{CGGN22} and the references therein for this notion). One only needs to replace the use of \cite[Proposition 5.5]{CGGN22} by Lemma \ref{lemma:isolated_singularities} below. Then \cite[Proposition 5.9]{CGGN22} applies to show that $Y$ admits a maximally quasi-\'etale cover. Therefore, replacing $X$ by a quasi-\'etale cover, if necessary, we may assume without loss of generality that $X$ is maximally quasi-\'etale.

Note that $X$ has only isolated quotient singularities by Lemma \ref{lemma:isolated_singularities} again. Moreover, $c_1(X)=0 \in H^2(X,\mathbb{C})$ by Theorem \ref{thm:vanishing}. By \cite[Corollary 4.2]{CGGN22}, there exists a complex torus $T$ as well as a compact K\"ahler variety $Z$ with canonical singularities and a quasi-\'etale cover $\gamma \colon T \times Z \to X$. In addition, we have $\omega_Z\cong \sO_Z$ and $\tilde{q}(Z)=0$. Note also that $\gamma$ is \'etale since $X$ is maximally quasi-\'etale by our current assumption, and hence $T \times Y$ has only isolated singularities. 

We now prove $\dim T \ge 1$. We argue by contradiction and assume $\dim T =0$. By Lemma \ref{lemma:flat_extension}, $T_{X_{\textup{reg}}}$ is projectively flat. Let 
\[
\rho \colon \pi_1\big(X_{\textup{reg}}\big) \to \textup{PGL}(n,\mathbb{C})
\]
be a representation that defines the projectively flat structure on $T_{X_{\textup{reg}}}$. Then $\rho$
factors through $\pi_1\big(X_{\textup{reg}}\big) \to \pi_1(X)$ (see \cite[Remark 5.4]{CGGN22}).
Moreover, the induced representation
\[
\pi_1(Z) \to \pi_1(X) \to \textup{PGL}(n,\mathbb{C}) 
\]
has finite image by \cite[Theorem 7.2]{CGGN22}. Replacing, $Z$ by a finite \'etale cover, if necessary, we may therefore assume that $\Omega_Z^{[1]} \cong \sL^{\oplus \dim Z}$, where $\sL$ is reflexive of rank $1$. Then 
$\sL^{[\otimes \dim Z]}\cong \sO_Z$ since $\omega_Z\cong \sO_Z$. Hence, replacing $Z$ by a further quasi-\'etale cover, we may assume that $\sL \cong \sO_Z$. This gives a contradiction since $\tilde{q}(Z)=0$ by assumption. 

Note that $Z$ is smooth since $T \times Z$ has only isolated singularities and $\dim T \ge 1$. By the Beauville - Bogomolov decomposition theorem, we may assume that $Z$ is simply connected. Let $t \in T$ and set $Z_t = \{t\}\times Z \subset T\times Z$. The tangent bundle $T_{T \times Z}$ is projectively flat with zero first Chern class by assumption. Therefore, its restriction to $Z_t$ is the trivial vector bundle. This in turn implies that $\Omega^1_{Z_t}$ is generated by its global sections, and hence $\dim Z_t =0$ since $q(Z_t)=0$.
This finishes the proof of the theorem.
\end{proof}

\begin{lemma}\label{lemma:isolated_singularities}
Let $X$ be a complex space with klt singularities. Suppose that the vector bundle $T_{X_{\textup{reg}}}$ is projectively flat. Suppose in addition that $X$ admits a maximally quasi-\'etale cover $\eta \colon Y \to X$. Then $Y$ has only isolated, cyclic quotient singularities.
\end{lemma}

\begin{proof}
Let $n$ be the dimension of $X$, and let 
\[
\pi_1\big(X_{\textup{reg}}\big) \to \textup{PGL}(n,\mathbb{C})
\]
be a representation that defines the projectively flat structure on $T_{X_{\textup{reg}}}$. 
By the Nagata-Zariski purity theorem, we have $\eta^{-1}\big(X_{\textup{reg}}\big) \subseteq Y_{\textup{reg}}$.
The inclusion now induces an isomorphism $\pi_1\big(\eta^{-1}\big(X_{\textup{reg}}\big)\big)\cong 
\pi_1\big(Y_{\textup{reg}}\big)$ of fundamental groups since $Y_{\textup{reg}}\setminus \eta^{-1}\big(X_{\textup{reg}}\big)$ is a closed subset of codimension at least $2$. The composition
\[
\pi_1\big(Y_{\textup{reg}}\big) \cong \pi_1\big(\eta^{-1}\big(X_{\textup{reg}}\big)\big) \to \pi_1\big(X_{\textup{reg}}\big) \to \textup{PGL}(n,\mathbb{C})
\]
then factors through $\pi_1\big(Y_{\textup{reg}}\big) \to \pi_1(Y)$ since $Y$ is maximally quasi-\'etale by assumption. Let $y \in Y$ be any point. Then there exists an open neighborhood $U$ of $y$ over which $\Omega_U^{[1]} \cong \sL^{\oplus n}$, where $\sL$ is reflexive of rank $1$. The claim now follows from \cite[Proposition 4.1]{GKP_proj_flat_JEP}.
\end{proof}

\begin{lemma}\label{lemma:flat_extension}
Let $X$ be a complex manifold, and let $\sE$ be a vector bundle on $X$. Let $X^\circ \subseteq X$ be a Zariski open subset with complement of codimension at least $2$. If $\sE|_{X^\circ}$ is flat $($resp. projectively flat$)$, then $\sE$ is flat $($resp. projectively flat$)$ as well.
\end{lemma}

\begin{proof}
Suppose that $\sE|_{X^\circ}$ is flat, and let $\nabla^\circ$ be a flat connection on $\sE|_{X^\circ}$. By the Hartogs extension theorem, $\nabla^\circ$ extends to a flat connection $\nabla$ on $\sE$. If $\sE|_{X^\circ}$ is projectively flat, then $\sE$ is projectively flat as well by Proposition \ref{prop:projectively_flat_versus_flat} together with the previous case.
\end{proof}

Next, we address transversely projectively flat foliations with compact leaves. Note that Theorem \ref{thm_intro:compact_leaves} is an immediate consequence of Theorem \ref{thm:compact_leaves} below together with Theorem \ref{thm:vanishing}.

\begin{thm}\label{thm:compact_leaves}
Let $X$ be a compact K\"ahler manifold, and let $\cF$ be regular foliation on $X$ with normal bundle $\sN$. Suppose that $\cF$ is transversely projectively flat with compact leaves. Suppose furthermore that $c_1(\sN)=0$. Then there exist a complex torus $T$ and a smooth morphism $f \colon Y \to T$ as well as a finite \'etale cover $\gamma\colon Y \to X$ such that $\gamma^{-1}\cF$ is induced by $f$. 
\end{thm}

\begin{proof}
The arguments used in the proof of \cite[Corollary 2.11]{hoering_split} show that $\cF$ is induced by an equidimensional morphism $f \colon X \to Y$ onto a compact analytic space with only quotient singularities. In particular, the complex variety $Y$ is $\mathbb{Q}$-factorial with klt singularities. Note also that any fiber of $f$ with reduced analytic structure is smooth since $\cF$ is regular by assumption.

Let $(B_i)_{i \in I}$ be the possibly empty set of prime divisors $B$ in $Y$ such that $f^*B$ is not integral. Then $f^*B_i=m_i D_i$ for some integer $m_i \ge 2$ and some prime divisor $D_i$. By the Reeb local stability theorem, the pair $\big(Y,\sum_{i\in I}\frac{m_i-1}{m_i}B_i\big)$ is klt. Moreover, a simple computation shows that $\det \sN^\vee \cong \sO_Y(f^*K_Y+\sum_{i\in I}(m_i-1) D_i)$, so that $K_Y+\sum_{i\in I}\frac{m_i-1}{m_i}B_i \equiv 0$. Then \cite[Corollary 1.18]{cao_guenancia_paun_2023} applies to show that $K_Y+\sum_{i\in I}\frac{m_i-1}{m_i}B_i$ is torsion. Let $\eta_1 \colon Y_1 \to Y$ be the index one cover of $\big(Y,\sum_{i\in I}\frac{m_i-1}{m_i}B_i\big)$
(see \cite[Section 2.4]{shokurov_log_flips}). By construction, we have
\[
K_{Y_1} \sim_\mathbb{Q} \eta_1^*\left(K_Y+\sum_{i\in I}\frac{m_i-1}{m_i}B_i\right) \sim_\mathbb{Q} 0.
\]
Let also $X_1$ be the normalization of the product $Y_1 \times_Y X$. A simple computation then shows that the natural morphism $\gamma_1 \colon X_1 \to X$ is \'etale in codimension $1$, and hence \'etale by purity of the branch locus.

Set $\cF_1=\gamma_1^{-1}\cF$. Then $\cF_1$ is a regular foliation with normal bundle $\sN_1 \cong \gamma_1^*\sN$. Moreover, $\cF_1$ is transversely projectively flat (see Example \ref{example:submersion}), and $c_1(\sN_1)=0$. By the Reeb local stability theorem again, the morphism $f_1\colon X_1 \to Y_1$ is smooth over an open subset $Y_1^\circ \subseteq Y_1$ with complement of codimension at least $2$. Hence, there is an identification $\sN_1|_{X_1^\circ} \cong (f_1|_{X_1^\circ})^* T_{Y_1^\circ}$, where $X_1^\circ:=f_1^{-1}(Y_1^\circ)$. This in turn implies that $T_{Y_1^\circ}$ is projectively flat by Lemma \ref{lemma:descent}. By Theorem \ref{thm:projectively_flat_spaces}, replacing $Y_1$ by a quasi-\'etale cover, we may assume that $Y_1$ is a complex torus. Then we have $\sN_1 \cong f_1^*T_{Y_1}$ since both are reflexives sheaves and agree away from a codimension $2$ subset. This implies that $f_1$ is a submersion since $\cF_1$ is a regular foliation, completing the proof of the proposition. 
\end{proof}

\section{Birational geometry of transversely projectively flat foliations}\label{section:birational_geometry}

In this section we study the birational geometry of transversely projectively flat foliations on complex projective manifolds. Let $\cF$ be a regular foliation on a complex projective manifold $X$ with normal bundle $\sN$. If $c_1(\sN)=0$, then $K_\cF\equiv K_X$ by the adjunction formula. By \cite{bchm}, we can run a $K_\cF$-MMP. 

\medskip

The following result extends \cite[Proposition 2.1 (1)]{jahnke_radloff_13} to leaf spaces of transversely projectively flat foliations. 

\begin{lemma}\label{lemma:simply_connected_subvarieties_tangent}
Let $\cF$ be regular foliation on a complex manifold $X$. Suppose that the normal bundle $\sN$ of $\cF$ is projectively flat with $c_1(\sN)=0$. Let $f \colon Y \to X$ be a holomorphic map from a simply connected compact K\"{a}hler manifold $Y$ to $X$. Then $f(Y)$ is tangent to $\cF$.
\end{lemma}

\begin{proof}
Let $q$ be the codimension of $\cF$. Note that $\sN|_Y \cong \sO_Y^{\oplus q}$ by our assumptions. The composition 
\[
T_Y \to T_X|_Y\to \sN|_Y\cong \sO_Y^{\oplus q}
\]
then vanishes identically since $h^0(Y,\Omega_Y^1) = h^1(X,\sO_X) = 0$ by Hodge theory. This proves that $f(Y)$ is tangent to $\cF$.
\end{proof}

Next, we describe extremal contractions of transversely projectively flat foliations of dimension at most $2$ on complex projective manifolds. 

\begin{prop}\label{prop:rational_leaves}
Let $\cF$ be regular foliation of dimension $1$ on a compact K\"ahler manifold $X$. Suppose that the normal bundle $\sN$ of $\cF$ is projectively flat with $c_1(\sN)=0$. If $X$ contains a rational curve $C$, then $\cF$ is induced by a $\mathbb{P}^1$-bundle structure $f\colon X \to Y$ onto a finite \'etale quotient of a complex torus. In addition, $C$ is a fiber of $f$.
\end{prop}

\begin{proof}
By Lemma \ref{lemma:simply_connected_subvarieties_tangent}, the curve $C$ is leaf of $\cF$. In particular, $C$ is a smooth rational curve. Then \cite[Corollary 2.11]{hoering_split} applies to show that $\cF$ is induced by a submersion $f \colon X \to Y$ onto a compact complex manifold. But then $\sN \cong f^*T_Y$. In particular, we have $c_1(Y)=0$ and $c_2(Y)=0$ by Fact \ref{fact:chern_classes}. By \cite[Chapter IV Corollary 4.15]{kobayashi_diff_geom_vb}, $Y$ is then covered by a complex torus. This finishes the proof of the proposition.
\end{proof}

\begin{prop}\label{prop:dimension_two_mmp}
Let $\cF$ be regular foliation of dimension $2$ on a complex projective manifold $X$. Suppose that the normal bundle $\sN$ of $\cF$ is projectively flat with $c_1(\sN)=0$. Suppose furthermore that $K_X$ is not nef, and let $f \colon X \to Y$ be $K_X$-negative extremal contraction. Then one of the following holds.
\begin{enumerate}
\item The map $f$ is a smooth morphism onto a finite \'etale quotient of an abelian variety, and the foliation $\cF$ is induced by $f$.
\item The variety $Y$ is smooth, and $f$ is the blow-up of $Y$ along a smooth subvariety $Z$ of codimension $2$. Moreover, $Z$ is a finite \'etale quotient of an abelian variety. In addition, there is a regular foliation $\cG$ on $Y$ whose normal bundle $\sN_\cG$ is projectively flat with $c_1(\sN_\cG)=0$ such that $Z$ is everywhere transverse to $\cG$ and $\cF=f^{-1}\cG$.
\item The morphism $f$ is a conic bundle with smooth discriminant locus $\Delta$. Moreover, any connected component of $\Delta$ is a finite \'etale quotient of an abelian variety. In addition, there is a regular foliation $\cH$ on $Y$ whose normal $\sN_\cH$ is projectively flat with $c_1(\sN_\cH)=0$ such that $\Delta$ is everywhere transverse to $\cH$ and $\cF=f^{-1}\cH$.
\end{enumerate}
\end{prop}

\begin{proof}
Let $n$ be the dimension of $X$. Any fiber of $f$ is rationally chain connected by \cite[Corollary 1.4]{hacon_mckernan}, and therefore contained in a leaf of $\cF$ by Lemma \ref{lemma:simply_connected_subvarieties_tangent}. This immediately implies $\dim X - \dim Y \le 2$. 

\medskip

If $\dim X - \dim Y = 2$, then $\cF$ is induced by $f$. Arguing as in the proof of Proposition \ref{prop:rational_leaves} above, we see that $f$ is a smooth morphism onto a finite \'etale quotient of an abelian variety.

\medskip

Suppose from now on $\dim X - \dim Y \le 1$. 

\medskip

Let $F$ be an irreducible component of a fiber of $f$. If $\dim F = 2$, then $F$ is a leaf of $\cF$. In particular, $F$ is smooth. By \cite[Theorem 1.9]{andreatta_wisniewski_view}, either $F$ is a projective plane, or $F$ is a Hirzebruch surface. In either case, $F$ is simply connected. Then \cite[Corollary 2.11]{hoering_split} applies to show that $\cF$ is induced by a submersion onto a complex projective manifold. Moreover, any leaf of $\cF$ is a fiber of $f$ by the Rigidity Lemma. In other words, $\cF$ is induced by $f$. This implies $\dim X - \dim Y = 2$, a contradiction. This shows that any fiber of $f$ has dimension at most $1$. By \cite[Theorem 1.2]{wisn_crelle}, we conclude that $Y$ is smooth, and either $f$ is the blow-up of $Y$ along a smooth subvariety $Z$ of codimension $2$, or $f$ is a conic bundle. 

\medskip

Suppose first that $f$ is the blow-up of $Y$ along a smooth subvariety $Z$ of codimension $2$. Let $E$ be its exceptional locus, and let $F\cong \mathbb{P}^1$ be any fiber of $f|_E \colon E \to Z$. Then $T_\cF|_F\cong \sO_{\mathbb{P}^1}(2)\oplus \sO_{\mathbb{P}^1}(-1)$ since $T_F \subset T_\cF$ by Lemma \ref{lemma:simply_connected_subvarieties_tangent} and $K_\cF \cdot F = K_X \cdot F = -1$. In particular, $E$ is everywhere transverse to $\cF$. Let $\cG$ be the (possibly singular) foliation on $Y$ induced by $\cF$. Since $c_1(\sN)=0$, there exists a line bundle $\sL$ on $Y$ such that $\det \sN \cong f^* \sL$. Note that $\sL \cong \det \sN_\cG $. Let $\omega \in H^0(Y,\Omega_Y^{n-2}\otimes \sL)$ be a twisted $(n-2)$-form defining $\cG$.
Then $f^*\omega \in H^0(X,\Omega_X^{n-2}\otimes \det \sN)$ is a twisted $(n-2)$-form defining $\cF$. This immediately implies that $\omega$ is nowhere vanishing. Therefore, $\cG$ is a regular foliation, and $\sN \cong f^*\sN_\cG$. Let 
\[
\rho \colon \pi_1(X) \to \textup{PGL}(n-2,\mathbb{C})
\]
be a representation that defines the projectively flat structure on $\sN$. Then $\rho$
factors through $\pi_1(X) \twoheadrightarrow \pi_1(Y)$, and the induced representation
\[
\rho \colon \pi_1(Y) \to \textup{PGL}(n-2,\mathbb{C})
\]
defines a projectively flat structure on $\sN_\cG$. Also, local calculations show that $Z$ is everywhere transverse to $\cG$ since $E$ is everywhere transverse to $\cF$. This implies that $T_Z \cong \sN_\cG|_Z$. In particular, we have $c_1(Z)=0$ and $c_2(Z)=0$ by Fact \ref{fact:chern_classes}. By \cite[Chapter IV Corollary 4.15]{kobayashi_diff_geom_vb}, $Z$ is a finite \'etale quotient of an abelian variety.

\medskip

Suppose finally that $f$ is a conic bundle structure on $X$ with discriminant locus $\Delta$. Let $F$ be an irreducible component of some fiber of $f$ over the discriminant locus. Then $F$ is a smooth rational curve with $-K_X \cdot F =1$. Moreover, $F$ is contained in a leaf of $\cF$ by Lemma \ref{lemma:simply_connected_subvarieties_tangent}. Thus $T_\cF|_F\cong \sO_{\mathbb{P}^1}(2)\oplus \sO_{\mathbb{P}^1}(-1)$ and $\sN|_F \cong \sO_{\mathbb{P}^1}^{\oplus n -2}$. This in turn implies 
$T_X|_F\cong \sO_{\mathbb{P}^1}(2)\oplus \sO_{\mathbb{P}^1}(-1)\oplus \sO_{\mathbb{P}^1}^{\oplus n -2}$. 
By \cite[Lemma 4.7]{druel_bbcd2}, we conclude that any fiber of $f$ is reduced and that $\Delta$ is smooth.
Moreover, $\cF$ is transverse to $f^{-1}(\Delta)$ at a general point of $F$.

By Lemma \ref{lemma:simply_connected_subvarieties_tangent} again, irreducible components of fibers of $f$ are tangent to $\cF$. This implies that there is a (possibly singular) foliation $\cH$ of dimension $1$ on $Y$ such that $\cF = f^{-1}\cH$. Since $c_1(\sN)=0$, there exists a line bundle $\sM$ on $Y$ such that $\det \sN \cong f^* \sM$. Note that $\cH$ is generically transverse to $\Delta$ since $\cF$ is transverse to $f^{-1}(\Delta)$ at a general point of $F$. This easily implies $\sM \cong \det \sN_\cH$. Let $\alpha \in H^0(Y,\Omega_Y^{n-2}\otimes \sM)$ be a twisted $(n-2)$-form defining $\cH$. Then $f^*\omega \in H^0(X,\Omega_X^{n-2}\otimes \det \sN)$ is a twisted $(n-2)$-form defining $\cF$. This immediately implies that $\alpha$ is nowhere vanishing. Thus $\cH$ is a regular foliation, and $\sN\cong f^*\sN_\cH$. Moreover, $\cH$ is everywhere transverse to $\Delta$ since $\cF$ is transverse to $f^{-1}(\Delta)$ at a general point of any fiber $F$ of $f$ over the discriminant locus. Let 
\[
\rho \colon \pi_1(X) \to \textup{PGL}(n-2,\mathbb{C})
\]
be a representation that defines the projectively flat structure on $\sN$. Then $\rho$
factors through $\pi_1(X) \twoheadrightarrow \pi_1(Y)$, and the induced representation
\[
\rho \colon \pi_1(Y) \to \textup{PGL}(n-2,\mathbb{C})
\]
defines a projectively flat structure on $\sN_\cH$. 
Note that $T_\Delta \cong \sN_\cH|_\Delta$ since $\cH$ is everywhere transverse to $\Delta$. Thus, we have $c_1(\Delta)=0$ and $c_2(\Delta)=0$ by Fact \ref{fact:chern_classes}.  By \cite[Chapter IV Corollary 4.15]{kobayashi_diff_geom_vb} again, any connected component of $\Delta$ is a finite \'etale quotient of an abelian variety. This finishes the proof of the proposition.
\end{proof}

Finally, we prove a special case of the abundance conjecture.

\begin{prop}\label{prop:abundance}
Let $\cF$ be regular foliation of dimension $p \ge 1$ on a complex projective manifold $X$. Suppose that the normal bundle $\sN$ of $\cF$ is projectively flat with $c_1(\sN)=0$, and that $K_X$ is nef. Suppose in addition that the abundance conjecture holds for complex projective manifolds of dimension at most $p-1$, and that $\cF$ is not algebraically integrable. Then $K_X$ is semiample.
\end{prop}

The following example shows to what extend the above result is optimal.

\begin{exmp}
Let $A$ be an abelian variety, and let $Y$ be a complex projective manifold. Let $\cF$ be the pull-back on $X:=Y \times A$ of a linear foliation of codimension at most $1$ on $A$. Then $\cF$ is transversely flat (see Examples \ref{example:linear_foliation} and \ref{example:submersion}). On the other hand, $K_X$ is semiample if and only if $K_Y$ is semiample.
\end{exmp}

\begin{proof}[Proof of Proposition \ref{prop:abundance}]
By \cite[Proposition 2.1]{kawamata}, in order to prove that $K_X$ is semiample, it suffices to prove that $K_X$ is abundant. 

Let $q$ be the codimension of $\cF$, and let
\[
\rho\colon \pi_1(X) \to \textup{PGL}(q,\mathbb{C})
\]
be a representation that defines the projectively flat structure on the normal bundle $\sN$ of $\cF$. Let $\textup{G} \subseteq \textup{PGL}(q,\mathbb{C})$ be the Zariski closure of $\rho(\pi_1(X))$. This is a linear algebraic group which has finitely many connected components. Applying Selberg's Lemma and passing to an appropriate finite \'etale cover of $X$, we may assume without loss of generality that $\textup{G}$ is connected and that the image of the induced representation
\[
\rho_1 \colon \pi_1(X) \to \textup{G} \to \textup{G}/\textup{Rad}\,\textup{G}
\]
is torsion-free, where $\textup{Rad}\,\textup{G}$ denotes the radical of $\textup{G}$. Let 
\[
\textup{sh}_{\rho_1}\colon X \dashrightarrow Y
\]
be the $\rho_1$-Shafarevich map, where $\textup{sh}_{\rho_1}$ is dominant and
$Y$ is a smooth projective variety. The rational map $\textup{sh}_{\rho_1}$ is almost proper with connected fibers. By \cite[Th\'eor\`eme 1]{CCE15}, we may assume without loss of generality that $Y$ is of general type, and that the representation $\rho_1$ factorizes through $\textup{sh}_{\rho_1}$. 

If $\dim Y = \dim X$, then $X$ is of general type and the claim follows from the base-point-free theorem. Suppose from now on that $\dim Y < \dim X$.

Let $F$ be a general (smooth) fiber of $\textup{sh}_{\rho_1}$. By \cite[Proposition 2.23]{druel_proj_flat}, in order to prove that $K_X$ is abundant, it suffices to prove that $K_F$ is abundant. Let $a \colon F \to \textup{A}$ be the Albanese morphism. By construction, we have $\dim \textup{A} = q(F)$. Let $F_1$ be a general (smooth) fiber of the Stein factorization of $F \to a(F)$. 
By \cite[Theorem 1.1]{hu_log_abundance}, in order to prove that $K_F$ is abundant, it suffices to prove that $K_{F_1}$ is abundant. In particular, it is enough to consider the case where $\dim F_1 >0$. Repeating the process finitely many times, we see that it suffices to prove that $K_{F_2}$ is abundant for some smooth projective subvariety $F_2 \subseteq F_1$ (passing through a general point) of dimension $\dim F_2 >0$ and irregularity $q(F_2)=0$. Let 
\[
\rho_2\colon \pi_1(F_2) \to \pi_1(X) \to \textup{PGL}(q,\mathbb{C}).
\]
be the induced representation. By construction, $\rho_2(\pi_1(F_2))$ is solvable and torsion-free. It follows  that $\rho_2$ is the trivial representation since $q(F_2)=0$. Therefore, there exists a line bundle 
$\sL_2$ on $F_2$ such that $\sN|_{F_2}\cong \sL_2^{\oplus q}$. Note that $c_1(\sL_2)=0$ since $c_1(\sN)=0$ by assumption, and hence the line bundle $\sL_2$ is torsion as $q(F_2)=0$. Let $F_3 \to F_2$ be the corresponding cyclic \'etale cover. Then $K_{F_2}$ is abundant if and only if $K_{F_3}$ is. Arguing as above, we see that it remains to prove that $K_{F_4}$ is abundant for some smooth projective subvariety (passing through a general point) $F_4 \subseteq F_3$ of dimension $\dim F_4 >0$ and irregularity $q(F_4)=0$. Now, the composition 
\[
T_{F_4} \to T_X|_{F_4}\to \sN|_{F_4}\cong \sO_{F_4}^{\oplus q} 
\]
vanishes identically, and hence $T_{F_4} \subseteq T_\cF|_{F_4}$. This implies $\dim F_4 \le p-1$ since $\cF$ is not algebraically integrable, and hence $K_{F_4}$ is abundant since the abundance conjecture holds for complex projective manifolds of dimension at most $p-1$ by assumption. This finishes the proof of the proposition.
\end{proof}

\section{Transversely flat foliations of dimension 1}\label{dimension_one}

In this section we describe the structure of regular foliations of dimension $1$ whose normal bundle is projectively flat with zero first Chern class. Before we give the proof of Theorem \ref{thm_intro:dimension_one}, we need the following auxiliary results.

\begin{prop}\label{prop:av}
Let $X$ be a compact K\"ahler manifold with $c_1(X)=0$, and let $\cF$ be regular foliation on $X$. Suppose that the normal bundle $\sN$ of $\cF$ is projectively flat with $c_1(\sN)=0$. Then there exists a complex torus $T$ as well as a simply connected compact K\"ahler manifold $Y$ with $c_1(Y)=0$, and a finite \'etale cover $\gamma \colon T \times Y \to X$ such that $\gamma^{-1}\cF$ is the pull-back of a linear foliation on $T$.
\end{prop}

\begin{proof}
By the Beauville - Bogomolov decomposition theorem, there exists a complex torus $T$ as well as a simply connected compact K\"ahler manifold $Y$ with $c_1(Y)=0$, and a finite \'etale cover $\gamma \colon T \times Y=:X_1 \to X$. The normal bundle $\sN_1$ of the regular foliation $\cF_1:= \gamma^{-1}\cF$ on $X_1$ satisfies $\sN_1 \cong \gamma^*\sN$. In particular, $\sN_1$ is projectively flat with $c_1(\sN_1)=0$. Let $t \in T$ be a point, and set $Y_t=\{t\} \times Y \subseteq T \times Y=X_1$. By Lemma \ref{lemma:simply_connected_subvarieties_tangent}, $Y_t$ is tangent to $\cF_1$. This implies that $\cF_1$ is the pull-back of a regular foliation $\cG$ on $T$ via the projection morphism $\textup{pr}_T\colon X_1 \to T$. Then 
$\sN_1 \cong \textup{pr}_T^{\,*}\sN_{\cG}$, and hence $c_1(\sN_{\cG})=0$. This easily implies that $\cG$ is a linear foliation, completing the proof of the proposition.
\end{proof}

\begin{lemma}\label{lemma:ab_schemes}
Let $f \colon X \to Y$ be an abelian scheme over a smooth projective base, and let $\cF$ be regular foliation with $T_\cF \subseteq T_{X/Y}$. Suppose that the normal bundle $\sN$ of $\cF$ is projectively flat with $c_1(\sN)=0$. Then there exists an abelian variety $A$ as well as a simply connected complex projective manifold $Z$ with $c_1(Z)=0$, and a finite \'etale cover $\gamma \colon A \times Z \to X$ such that $\gamma^{-1}\cF$ is the pull-back of a linear foliation on $A$. In particular, we have $c_1(X)=0$.
\end{lemma}

\begin{proof}
Let $\Sigma\subseteq X$ be the neutral section of $f$. Note that 
\[
\sN|_\Sigma \cong (T_{X/Y}/ T_\cF)|_\Sigma \oplus T_\Sigma
\]
since the exact sequence 
\[
0 \to T_{X/Y}|_\Sigma \to T_X|_\Sigma \to f^*T_Y|_\Sigma \cong T_\Sigma \to 0
\]
is split and $T_\cF\subseteq T_{X/Y}$ by assumption. By Lemma \ref{lemma:first_chern_class_direct_summand}, we have $c_1(Y)=0$. Replacing $X$ by a finite \'etale cover, if necessary, we may assume without loss of generality that $f$ is equipped with a level three structure. Let $\sA(3)$ be the fine moduli space of polarized abelian varieties with a level three structure. The moduli map $Y \to \sA(3)$ is then the constant morphism by \cite[Lemma 5.9.3]{kollar_sh_inventiones}. In particular, we have $c_1(X)=0$. The claim now follows from 
Proposition \ref{prop:av} above.
\end{proof}

Finally, we describe the structure of transversely projectively flat foliations of dimension $1$ on complex projective manifold. Note that Theorem \ref{thm_intro:dimension_one} is a consequence of Theorem \ref{thm:dimension_one} below together with Theorem \ref{thm:vanishing}.

\begin{thm}\label{thm:dimension_one}
Let $\cF$ be regular foliation of dimension $1$ on a complex projective manifold $X$. Suppose that the normal bundle $\sN$ of $\cF$ is projectively flat with $c_1(\sN)=0$. Then one of the following holds.
\begin{enumerate}
\item The foliation $\cF$ is induced by a $\mathbb{P}^1$-bundle structure $f\colon X \to Y$ onto a finite \'etale quotient of an abelian variety.
\item There exists an abelian variety $A$ as well as a finite \'etale morphism $\gamma\colon A \to X$ such that $\gamma^{-1}\cF$ is a linear foliation.
\item There exist a smooth projective curve $C$ and an abelian variety $B$ as well as a finite \'etale mrophism $\gamma\colon C \times B \to X$ such that $\gamma^{-1}\cF$ induces a flat Ehresmann connection on the projection morphism $\textup{pr}_C \colon C \times B \to C$.
\end{enumerate}
\end{thm}

\begin{proof}
Suppose first that $K_X$ is not nef. Then $X$ contains a rational curve by the cone theorem and Proposition \ref{prop:rational_leaves} applies to show that $\cF$ is induced by a $\mathbb{P}^1$-bundle structure $f \colon X \to Y$, where $Y$ is a finite \'etale quotient of an abelian variety. 

Suppose from now on that $K_X$ is nef. Note that $K_\cF \equiv K_X$ by the adjunction formula.
The exact sequence
\[
0 \to T_\cF \to T_X \to \sN \to 0 
\]
gives $c_2(X) = 0$ since $c_1(\sN)=0$ and $c_2(\sN)=0$ (see Fact \ref{fact:chern_classes}). By \cite[Theorem 1.2]{iwai_matsumura_muller_2025}, either $X$ is a finite \'etale quotient of an abelian variety, or there exists an abelian scheme $Z \to C$ over a smooth projective curve $C$ of genus at least $2$ as well as a finite \'etale morphism $\gamma\colon Z \to X$.

If $X$ is a finite \'etale quotient of an abelian variety, then Poposition \ref{prop:av} easily implies that we are in case (2) of Theorem \ref{thm_intro:dimension_one}.

Suppose finally that there exist an abelian scheme $Z \to C$ over a smooth projective curve $C$ of genus at least $2$ and a finite \'etale morphism $\gamma\colon Z \to X$. Set $\cG:=\gamma^{-1}\cF$. Lemma \ref{lemma:ab_schemes} then implies that $T_\cG \not\subseteq T_{Z/C}$. Let $F$ be a general fiber of $f$. Then $K_\cG|_F \equiv K_Z|_F \equiv 0$, and therefore $\cG$ is transverse to $f$ at any point in $F$. This in turn implies that $f$ is isotrivial. Replacing $Z$ by a further finite \'etale cover, we may therefore assume that $Z/C=C \times B/C$, where $B$ is an abelian variety. Note also that
\[
K_\cG \equiv K_{C \times B} \equiv \textup{pr}_C^* K_C.
\]
This immediately implies that $\cG$ induces a (flat) Ehresmann connection on the projection morphism $\textup{pr}_C \colon C \times B \to C$, completing the proof of the theorem.
\end{proof}

The following is an immediate consequence of Theorem \ref{thm_intro:dimension_one}. 

\begin{cor}
Let $\cF$ be regular foliation of dimension $1$ on a complex projective manifold $X$. Suppose that the normal bundle $\sN$ of $\cF$ is projectively flat with $c_1(\sN)=0$. Then there exists a finite \'etale cover $\gamma \colon Y \to X$ such that $\gamma^{-1}\cF$ is defined by $\dim X - 1$ pointwise linearly independent holomorphic $1$-forms.
\end{cor}

\begin{question}\label{question:abundance}
Let $\cF$ be regular foliation on a complex projective manifold $X$ with normal bundle $\sN$. Suppose that $\cF$ is transversely projectively flat with $c_1(\sN)=0$. One may ask if there exists a finite \'etale cover $\eta\colon Y \to X$ such that $\eta^* \sN$ is the trivial vector bundle. 
\end{question}

\begin{rem}
Suppose the answer to Question \ref{question:abundance} is yes if $p=2$. Then Theorem \ref{thm_intro:vanishing} and Propositions \ref{prop:dimension_two_mmp} and \ref{prop:abundance} together with the results of Hao, Wang and Zhang \cite{hao_al} on good minimal models with nowhere vanishing holomorphic $1$-forms (see also \cite{church}) then give the structure of transversely projectively flat foliations of dimension $2$ on complex projective manifolds of dimension at least $4$.
\end{rem}

Finally, we prove a special case of an analogue of the abundance conjecture for the space of leaves of a codimension 1 regular foliation. This shows that the answer to Question \ref{question:abundance} is yes if the codimension of $\cF$ is $1$.

\begin{thm}
Let $\cF$ be regular foliation of codimension $1$ on a complex projective manifold $X$. Suppose that the first Chern class of the normal bundle $\sN$ of $\cF$ vanishes. Then $\sN$ is a torsion line bundle.
\end{thm}

\begin{proof}
Let $\nabla$ be a (holomorphic) unitary flat connection on $\sN$, and let 
\[\chi \colon \pi_1(X) \to \mathbb{C^*}
\] 
be its monodromy representation. Note that $\sN$ is torsion if and only if $\chi (\pi_1(X))$ is finite. We will denote by $\mathbb{C}_\chi$ the $\pi_1(X)$-module $\mathbb{C}$ equipped with the action defined by $\chi$.

Let now $\omega \in H^0(X,\Omega_X^1\otimes \sN)$ be a twisted $1$-form defining $\cF$. Then $\nabla^1(\omega)=0 \in H^0(X,\Omega_X^2\otimes \sN)$ by Hodge theory for unitary local systems. As a consequence, there exists 
an open covering $(U_i)_{i \in I}$ of $X$ as well as submersions $f_i \colon U_i \to \mathbb{C}$ and flat sections $s_i$ of $\sN|_{U_i}$ such that $\omega|_{U_i}=\dd f_i \otimes s_i$. The $f_i$ define $\cF$ and satisfy $f_i=a_{ij}f_j+b_{ij}$ on $U_i\cap U_j$ with $a_{ij} \in \mathbb{C}^*$ and $b_{ij} \in \mathbb{C}$. In other words, $\cF$ is transversely affine. Let $i_0 \in I$ and let
\[
\rho \colon \pi_1(X) \to \textup{Aff}(\mathbb{C})
\]
be the representation obtained by analytic continuation of $f_{i_0}$ along loops based at some point $x_0 \in U_{i_0}$, where $\textup{Aff}(\mathbb{C})$ is the group of affine automorphism of $\mathbb{C}$. There exists a $1$-cocycle $\tau \colon \pi_1(X) \to \mathbb{C}$ such that, for any $\gamma \in \pi_1(X)$ and any $z \in \mathbb{C}$, we have 
\[
\rho(\gamma)(z)=\chi^{-1}(\gamma)z+\tau(\gamma).
\]
Suppose that $[\tau]=0 \in H^1(\pi_1(X),\mathbb{C}_\chi)$. In other words, $\tau$ is a $1$-coboundary. Then $\rho$ fixes a point $z_0 \in \mathbb{C}$. This implies that $\cF$ has a developing map $X \to \mathbb{C}$, giving a contradiction since $X$ is compact by assumption. This shows that $[\tau] \neq 0 \in H^1(\pi_1(X),\mathbb{C}_\chi)$.

We argue by contradiction and assume that $\chi$ is not torsion. By \cite[Theorem 5.1]{bartolo_al_2013}, $\rho$ factors through a map $f \colon X \to Y$ to an orbifold $Y$ of dimension one. This implies that $\cF$ is algebraically integrable. Moreover, there is an identification $\sN \cong f^*T_Y$. It follows that either $Y$ is a (smooth) projective curve of genus $1$ or the coarse moduli space of $Y$ is isomorphic to $\mathbb{P}^1$ since $c_1(\sN)=0$ by assumption. In either case, $\sN$ is torsion. This gives a contradiction, finishing the proof of the theorem.
\end{proof}

\begin{question}\label{question:abundance_2}
Let $\cF$ be regular foliation on a complex projective manifold $X$ with normal bundle $\sN$. Suppose $c_1(\sN)=0$. One may ask if $\det \sN$ is torsion. 
\end{question}

\bibliographystyle{amsalpha}
\bibliography{foliation}

@UNPUBLISHED {church,
    AUTHOR = {Church, Benjamin},
     TITLE = {Nowhere vanishing 1-forms on varieties admitting a good minimal model},
      NOTE = {preprint {\tt arXiv:2410.22753}},
      YEAR = {2024},
 }

@UNPUBLISHED {hao_al,
    AUTHOR = {Hao, Feng and Wang, Zichang and Zhang, Lei},
     TITLE = {Good Minimal Models with Nowhere Vanishing Holomorphic 1-forms},
      NOTE = {Ann. Sc. Norm. Super. Pisa Cl. Sci., to appear},
 }

@article {bartolo_al_2013,
    AUTHOR = {Artal Bartolo, Enrique and Cogolludo-Agust\'in, Jos\'e{}
              Ignacio and Matei, Daniel},
     TITLE = {Characteristic varieties of quasi-projective manifolds and
              orbifolds},
   JOURNAL = {Geom. Topol.},
  FJOURNAL = {Geometry \& Topology},
    VOLUME = {17},
      YEAR = {2013},
    NUMBER = {1},
     PAGES = {273--309},
      ISSN = {1465-3060,1364-0380},
   MRCLASS = {14F35 (14B05 32S20 32S50 58K65)},
MRREVIEWER = {Dmitry\ Kerner},
       DOI = {10.2140/gt.2013.17.273},
       URL = {https://doi.org/10.2140/gt.2013.17.273},
}

@article {biswas_semistability,
    AUTHOR = {Biswas, Indranil},
     TITLE = {Semistability and restrictions of tangent bundle to curves},
   JOURNAL = {Geom. Dedicata},
  FJOURNAL = {Geometriae Dedicata},
    VOLUME = {142},
      YEAR = {2009},
     PAGES = {37--46},
      ISSN = {0046-5755,1572-9168},
   MRCLASS = {14F05 (32L10)},
MRREVIEWER = {Graeme\ Wilkin},
       DOI = {10.1007/s10711-009-9356-3},
       URL = {https://doi.org/10.1007/s10711-009-9356-3},
}

@article{iwai_matsumura_muller_2025,
author = {Iwai, Masataka and Matsumura, Shin-ichi and Müller, Niklas},
title = {Minimal projective varieties satisfying Miyaoka's equality},
journal = {Proceedings of the London Mathematical Society},
volume = {131},
number = {6},
pages = {e70104},
doi = {https://doi.org/10.1112/plms.70104},
url = {https://londmathsoc.onlinelibrary.wiley.com/doi/abs/10.1112/plms.70104},
eprint = {https://londmathsoc.onlinelibrary.wiley.com/doi/pdf/10.1112/plms.70104},
year = {2025}
}

@article {cao_guenancia_paun_2023,
    AUTHOR = {Cao, Junyan and Guenancia, Henri and P{\u a}un, Mihai},
     TITLE = {Variation of singular {K}\"ahler-{E}instein metrics: {K}odaira
              dimension zero},
      NOTE = {With an appendix by Valentino Tosatti},
   JOURNAL = {J. Eur. Math. Soc. (JEMS)},
  FJOURNAL = {Journal of the European Mathematical Society (JEMS)},
    VOLUME = {25},
      YEAR = {2023},
    NUMBER = {2},
     PAGES = {633--679},
      ISSN = {1435-9855,1435-9863},
   MRCLASS = {14J10 (14E30 14J32 32Q20)},
MRREVIEWER = {Ruadha\'i\ Dervan},
       DOI = {10.4171/jems/1184},
       URL = {https://doi-org.ezproxy.math.cnrs.fr/10.4171/jems/1184},
}

@article {langer_simpson,
    AUTHOR = {Langer, Adrian and Simpson, Carlos},
     TITLE = {Rank 3 rigid representations of projective fundamental groups},
   JOURNAL = {Compos. Math.},
  FJOURNAL = {Compositio Mathematica},
    VOLUME = {154},
      YEAR = {2018},
    NUMBER = {7},
     PAGES = {1534--1570},
      ISSN = {0010-437X,1570-5846},
   MRCLASS = {14F35 (14D06 14D07 14E20)},
MRREVIEWER = {Azniv\ Kasparian},
       DOI = {10.1112/s0010437x18007182},
       URL = {https://doi-org.ezproxy.math.cnrs.fr/10.1112/s0010437x18007182},
}

@article {CGGN22,
    AUTHOR = {Claudon, Beno\^{i}t and Graf, Patrick and Guenancia, Henri and
              Naumann, Philipp},
     TITLE = {K\"{a}hler spaces with zero first {C}hern class: {B}ochner
              principle, {A}lbanese map and fundamental groups},
   JOURNAL = {J. Reine Angew. Math.},
  FJOURNAL = {Journal f\"{u}r die Reine und Angewandte Mathematik. [Crelle's
              Journal]},
    VOLUME = {786},
      YEAR = {2022},
     PAGES = {245--275},
      ISSN = {0075-4102,1435-5345},
   MRCLASS = {53B35},
MRREVIEWER = {Marta\ Teofilova},
       DOI = {10.1515/crelle-2022-0001},
       URL = {https://doi.org/10.1515/crelle-2022-0001},
}

@article {druel_IMJ,
    AUTHOR = {Druel, St\'ephane},
     TITLE = {Projectively flat foliations},
   JOURNAL = {J. Inst. Math. Jussieu},
  FJOURNAL = {Journal of the Institute of Mathematics of Jussieu. JIMJ.
              Journal de l'Institut de Math\'ematiques de Jussieu},
    VOLUME = {24},
      YEAR = {2025},
    NUMBER = {3},
     PAGES = {763--812},
      ISSN = {1474-7480,1475-3030},
   MRCLASS = {37F75 (14E20 14E30 32Q26 32Q30)},
       DOI = {10.1017/S1474748024000574},
       URL = {https://doi.org/10.1017/S1474748024000574},
}

@article {hoering_split,
    AUTHOR = {H\"oring, Andreas},
     TITLE = {Uniruled varieties with split tangent bundle},
   JOURNAL = {Math. Z.},
  FJOURNAL = {Mathematische Zeitschrift},
    VOLUME = {256},
      YEAR = {2007},
    NUMBER = {3},
     PAGES = {465--479},
      ISSN = {0025-5874,1432-1823},
   MRCLASS = {32J27 (32Q30)},
MRREVIEWER = {M.\ G.\ Soares},
       DOI = {10.1007/s00209-006-0072-5},
       URL = {https://doi.org/10.1007/s00209-006-0072-5},
}

@book {molino,
    AUTHOR = {Molino, Pierre},
     TITLE = {Riemannian foliations},
    SERIES = {Progress in Mathematics},
    VOLUME = {73},
      NOTE = {Translated from the French by Grant Cairns,
              With appendices by Cairns, Y. Carri\`ere, \'E. Ghys, E. Salem
              and V. Sergiescu},
 PUBLISHER = {Birkh\"auser Boston, Inc., Boston, MA},
      YEAR = {1988},
     PAGES = {xii+339},
      ISBN = {0-8176-3370-7},
   MRCLASS = {53C12 (57R30 58A35 58F18)},
MRREVIEWER = {James\ J.\ Hebda},
       DOI = {10.1007/978-1-4684-8670-4},
       URL = {https://doi.org/10.1007/978-1-4684-8670-4},
}

@book {moerdijk_mrcun_foliations,
    AUTHOR = {Moerdijk, I. and Mr{\v{c}}un, J.},
     TITLE = {Introduction to foliations and {L}ie groupoids},
    SERIES = {Cambridge Studies in Advanced Mathematics},
    VOLUME = {91},
 PUBLISHER = {Cambridge University Press, Cambridge},
      YEAR = {2003},
     PAGES = {x+173},
      ISBN = {0-521-83197-0},
   MRCLASS = {58H05 (17B99 57R30)},
MRREVIEWER = {Jan\ Kubarski},
       DOI = {10.1017/CBO9780511615450},
       URL = {https://doi.org/10.1017/CBO9780511615450},
}

@article {aubin_KE,
    AUTHOR = {Aubin, Thierry},
     TITLE = {\'{E}quations du type {M}onge-{A}mp\`ere sur les
              vari\'{e}t\'{e}s k\"{a}hleriennes compactes},
   JOURNAL = {C. R. Acad. Sci. Paris S\'{e}r. A-B},
  FJOURNAL = {Comptes Rendus Hebdomadaires des S\'{e}ances de l'Acad\'{e}mie
              des Sciences. S\'{e}ries A et B},
    VOLUME = {283},
      YEAR = {1976},
    NUMBER = {3},
     PAGES = {Aiii, A119--A121},
      ISSN = {0151-0509},
   MRCLASS = {58G99 (32C10 35J60)},
}

@article {Yau_KE,
    AUTHOR = {Yau, Shing Tung},
     TITLE = {Calabi's conjecture and some new results in algebraic
              geometry},
   JOURNAL = {Proc. Nat. Acad. Sci. U.S.A.},
  FJOURNAL = {Proceedings of the National Academy of Sciences of the United
              States of America},
    VOLUME = {74},
      YEAR = {1977},
    NUMBER = {5},
     PAGES = {1798--1799},
      ISSN = {0027-8424},
   MRCLASS = {53C55 (14M20 32C10)},
       DOI = {10.1073/pnas.74.5.1798},
       URL = {https://doi.org/10.1073/pnas.74.5.1798},
}

@article {druel_proj_flat,
    AUTHOR = {Druel, St\'ephane},
     TITLE = {Projectively flat log smooth pairs},
   JOURNAL = {Ann. Fac. Sci. Toulouse Math. (6)},
  FJOURNAL = {Annales de la Facult\'e{} des Sciences de Toulouse.
              Math\'ematiques. S\'erie 6},
    VOLUME = {33},
      YEAR = {2024},
    NUMBER = {3},
     PAGES = {611--645},
      ISSN = {0240-2963,2258-7519},
   MRCLASS = {14E20 (14E30 32Q26 32Q30 53B10)},
MRREVIEWER = {Rong\ Du},
       DOI = {10.5802/afst.1783},
       URL = {https://doi.org/10.5802/afst.1783},
}

@article {kollar_sh_inventiones,
    AUTHOR = {Koll{\'a}r, J{\'a}nos},
     TITLE = {Shafarevich maps and plurigenera of algebraic varieties},
   JOURNAL = {Invent. Math.},
  FJOURNAL = {Inventiones Mathematicae},
    VOLUME = {113},
      YEAR = {1993},
    NUMBER = {1},
     PAGES = {177--215},
      ISSN = {0020-9910},
   MRCLASS = {14E20 (14E30 14J10)},
MRREVIEWER = {Alessio Corti},
       DOI = {10.1007/BF01244307},
       URL = {https://doi.org/10.1007/BF01244307},
}

@article {CCE15,
    AUTHOR = {Campana, Fr\'{e}deric and Claudon, Beno\^{\i}t and Eyssidieux,
              Philippe},
     TITLE = {Repr\'{e}sentations lin\'{e}aires des groupes k\"{a}hl\'{e}riens:
              factorisations et conjecture de {S}hafarevich lin\'{e}aire},
   JOURNAL = {Compos. Math.},
  FJOURNAL = {Compositio Mathematica},
    VOLUME = {151},
      YEAR = {2015},
    NUMBER = {2},
     PAGES = {351--376},
      ISSN = {0010-437X},
   MRCLASS = {32Q15 (14D07 14E20 14F35 32Q30)},
       DOI = {10.1112/S0010437X14007751},
       URL = {https://doi.org/10.1112/S0010437X14007751},
}

@article {jahnke_radloff_13,
    AUTHOR = {Jahnke, Priska and Radloff, Ivo},
     TITLE = {Semistability of restricted tangent bundles and a question of
              {I}. {B}iswas},
   JOURNAL = {Internat. J. Math.},
  FJOURNAL = {International Journal of Mathematics},
    VOLUME = {24},
      YEAR = {2013},
    NUMBER = {1},
     PAGES = {1250122, 15},
      ISSN = {0129-167X},
   MRCLASS = {32Q26 (14E20 14E30 32J27 32L10)},
       DOI = {10.1142/S0129167X12501224},
       URL = {https://doi.org/10.1142/S0129167X12501224},
}

@article {hu_log_abundance,
    AUTHOR = {Hu, Zhengyu},
     TITLE = {Log canonical pairs over varieties with maximal {A}lbanese
              dimension},
   JOURNAL = {Pure Appl. Math. Q.},
  FJOURNAL = {Pure and Applied Mathematics Quarterly},
    VOLUME = {12},
      YEAR = {2016},
    NUMBER = {4},
     PAGES = {543--571},
      ISSN = {1558-8599},
   MRCLASS = {14E30},
MRREVIEWER = {Kenta Hashizume},
       DOI = {10.4310/PAMQ.2016.v12.n4.a5},
       URL = {https://doi.org/10.4310/PAMQ.2016.v12.n4.a5},
}

@article {GKP_proj_flat_JEP,
    AUTHOR = {Greb, Daniel and Kebekus, Stefan and Peternell, Thomas},
     TITLE = {Projectively flat klt varieties},
   JOURNAL = {J. \'{E}c. polytech. Math.},
  FJOURNAL = {Journal de l'\'{E}cole polytechnique. Math\'{e}matiques},
    VOLUME = {8},
      YEAR = {2021},
     PAGES = {1005--1036},
      ISSN = {2429-7100},
   MRCLASS = {14E30 (14E20 32Q26 32Q30 53B10)},
       DOI = {10.5802/jep.164},
       URL = {https://doi.org/10.5802/jep.164},
}

@book {kobayashi_diff_geom_vb,
    AUTHOR = {Kobayashi, Shoshichi},
     TITLE = {Differential geometry of complex vector bundles},
    SERIES = {Publications of the Mathematical Society of Japan},
    VOLUME = {15},
      NOTE = {Kan\^{o} Memorial Lectures, 5},
 PUBLISHER = {Princeton University Press, Princeton, NJ; Princeton
              University Press, Princeton, NJ},
      YEAR = {1987},
     PAGES = {xii+305},
      ISBN = {0-691-08467-X},
   MRCLASS = {53C55 (32-02 32L05 32L10 32L20)},
       DOI = {10.1515/9781400858682},
       URL = {https://doi.org/10.1515/9781400858682},
}

@article {shokurov_log_flips,
    AUTHOR = {Shokurov, V. V.},
     TITLE = {Three-dimensional log perestroikas},
   JOURNAL = {Izv. Ross. Akad. Nauk Ser. Mat.},
  FJOURNAL = {Rossi\u{\i}skaya Akademiya Nauk. Izvestiya. Seriya
              Matematicheskaya},
    VOLUME = {56},
      YEAR = {1992},
    NUMBER = {1},
     PAGES = {105--203},
      ISSN = {1607-0046},
   MRCLASS = {14E05 (14E35)},
       URL = {https://doi.org/10.1070/IM1993v040n01ABEH001862},
}

@article {GL,
    AUTHOR = {Gongyo, Yoshinori and Lehmann, Brian},
     TITLE = {Reduction maps and minimal model theory},
   JOURNAL = {Compos. Math.},
  FJOURNAL = {Compositio Mathematica},
    VOLUME = {149},
      YEAR = {2013},
    NUMBER = {2},
     PAGES = {295--308},
      ISSN = {0010-437X},
   MRCLASS = {14E30},
}

@incollection {andreatta_wisniewski_view,
    AUTHOR = {Andreatta, Marco and Wi{\'s}niewski, Jaros{\l}aw A.},
     TITLE = {A view on contractions of higher-dimensional varieties},
 BOOKTITLE = {Algebraic geometry---{S}anta {C}ruz 1995},
    SERIES = {Proc. Sympos. Pure Math.},
    VOLUME = {62},
     PAGES = {153--183},
 PUBLISHER = {Amer. Math. Soc., Providence, RI},
      YEAR = {1997},
   MRCLASS = {14E30 (14J45)},
MRREVIEWER = {M. Kh. Gizatullin},
}

@article {atiyah57,
    AUTHOR = {Atiyah, M. F.},
     TITLE = {Complex analytic connections in fibre bundles},
   JOURNAL = {Trans. Amer. Math. Soc.},
  FJOURNAL = {Transactions of the American Mathematical Society},
    VOLUME = {85},
      YEAR = {1957},
     PAGES = {181--207},
      ISSN = {0002-9947},
   MRCLASS = {53.3X},
MRREVIEWER = {F. Hirzebruch},
}

@article {bchm,
    AUTHOR = {Birkar, Caucher and Cascini, Paolo and Hacon, Christopher D.
              and McKernan, James},
     TITLE = {Existence of minimal models for varieties of log general type},
   JOURNAL = {J. Amer. Math. Soc.},
  FJOURNAL = {Journal of the American Mathematical Society},
    VOLUME = {23},
      YEAR = {2010},
    NUMBER = {2},
     PAGES = {405--468},
      ISSN = {0894-0347},
   MRCLASS = {14E30 (14E05)},
MRREVIEWER = {Mark Gross},
       DOI = {10.1090/S0894-0347-09-00649-3},
       URL = {http://dx.doi.org/10.1090/S0894-0347-09-00649-3},
}

@article {druel_bbcd2,
    AUTHOR = {Druel, St\'ephane},
     TITLE = {Some remarks on regular foliations with numerically trivial canonical class},
   JOURNAL = {EPIGA},
  FJOURNAL = {EPIGA},
    VOLUME = {1},
      YEAR = {2017},
}

@article {hacon_mckernan,
    AUTHOR = {Hacon, Christopher D. and McKernan, James},
     TITLE = {On {S}hokurov's rational connectedness conjecture},
   JOURNAL = {Duke Math. J.},
  FJOURNAL = {Duke Mathematical Journal},
    VOLUME = {138},
      YEAR = {2007},
    NUMBER = {1},
     PAGES = {119--136},
      ISSN = {0012-7094},
     CODEN = {DUMJAO},
   MRCLASS = {14E30 (14E05 14J45)},
MRREVIEWER = {Mihnea Popa},
       DOI = {10.1215/S0012-7094-07-13813-4},
       URL = {http://dx.doi.org/10.1215/S0012-7094-07-13813-4},
}

@article {kawamata,
    AUTHOR = {Kawamata, Y.},
     TITLE = {Pluricanonical systems on minimal algebraic varieties},
   JOURNAL = {Invent. Math.},
  FJOURNAL = {Inventiones Mathematicae},
    VOLUME = {79},
      YEAR = {1985},
    NUMBER = {3},
     PAGES = {567--588},
      ISSN = {0020-9910},
     CODEN = {INVMBH},
   MRCLASS = {14C20 (14J10)},
       DOI = {10.1007/BF01388524},
       URL = {http://dx.doi.org/10.1007/BF01388524},
}

@article{wisn_crelle,
    author = {J.~A.~Wi{\'s}niewski},
    title  = {On Contractions of Extremal Rays of {F}ano manifolds},
    journal = {J.~Reine Angew.~Math.},
    year = 1991,
    volume = 417,
    pages = {141--157}
 }
\end{document}